 \documentclass[preprint,12pt]{elsarticle}

\usepackage{amsmath,amsthm,amssymb, epsf,
  amsfonts}
\newtheorem{theorem}{Theorem} \newtheorem{remark}{Remark}
 \newcommand{\bH}{\mathbf H}\newcommand{\bm}{\mathbf
  m}\newcommand{\bx}{\mathbf x}\newcommand{\by}{\mathbf
  y}\newcommand{\bw}{\mathbf w}\newcommand{\bz}{\mathbf
  z}\newcommand{\bn}{\mathbf n} \newcommand{\bu}{\mathbf u}
 \newcommand{\buep}{\mathbf
  u_{\varepsilon}} \newcommand{\ep}{\varepsilon}
\newcommand{\omC}{\Omega\setminus\cup\calP_i^\ep}
\newcommand{\bU}{\mathbf U}

 \newcommand{\calP}{\mathcal P}\newtheorem{lemma}{Lemma}  
\begin{document}

\begin{frontmatter}

\title{Effective models for nematic liquid
  crystals composites with ferromagnetic inclusions} 
\author{M.~C.~Calderer} \ead{mcc@math.umn.edu}\address{School of Mathematics, University of Minnesota, 507 Church Street S.E. 
Minneapolis, MN 55455}
\author{A.~DeSimone}\ead{desimone@sissa.it}\address{SISSA-International School for Advanced Studies
Via Beirut 2-4, 34014, Trieste, Italy}
\author{D.~Golovaty}\ead{dmitry@uakron.edu} \address{Department of Mathematics, The University of Akron, Akron,OH
  44325-4002} 
\author{A.~Panchenko}\ead{panchenko@math.wsu.edu}\address{Department of Mathematics, Washington State University, Pullman, WA 99164
}

\begin{abstract}
  Molecules of a nematic liquid crystal respond to an applied magnetic field by reorienting themselves in the direction of the field.  Since the dielectric anisotropy of a nematic is small, it takes relatively large fields to elicit a significant liquid crystal response.  The interaction may be enhanced in colloidal suspensions of ferromagnetic particles in a liquid crystalline matrix---ferronematics--- as proposed by Brochard and de Gennes in 1970. The ability of these particles to align with the field and, simultaneously, cause reorientation of the nematic molecules, greatly increases the magnetic response of the mixture. Essentially the particles provide an easy axis of magnetization that interacts with the liquid crystal via surface anchoring.

  We derive an expression for the effective energy of ferronematic in the dilute limit, that is, when the number of particles tends to infinity while their total volume fraction tends to zero. The total energy of the mixture is assumed to be the sum of the bulk elastic liquid crystal contribution, the anchoring energy of the liquid crystal on the surfaces of the particles, and the magnetic energy of interaction between the particles and the applied magnetic field.  The homogenized limiting ferronematic energy is obtained rigorously using a variational approach. It generalizes formal expressions previously reported in a physical literature.

\end{abstract}
\end{frontmatter}

\section{Introduction}

The study of magnetic particle suspensions in a liquid crystalline matrix was initiated with the theoretical article by Brochard and de~Gennes \cite{BdG70} (July, 1970), and the experimental work carried out by Rault, Cladis and Burger \cite{RCB70}, (June, 1970). {\footnote{Both groups acknowledge an ongoing scientific communication while their works were underway.}}.  The underlying mechanism behind a ferronematic system is a mechanical coupling between the nematic molecules and the magnetic particles, mostly realized by the surface anchoring energy.


Molecules of nematic liquid crystals have positive magnetic susceptibility, so they  tend to align themselves in the direction of an applied magnetic field. However, since this magnetic susceptibility is small---of order $10^{-7}$---it takes large fields, about $10^4$ Oe, to elicit a significant response.   Brochard and de~Gennes argue that the addition of paramagnetic ions to the system is not an efficient way to increase the  magnetic susceptibility constant, since it would require a concentration of paramagnetic ions  above $n=10^{20}$ ions per $\textrm{cm}^3$. The latter is the limiting  value that cannot be exceeded in order to guarantee the preservation of the liquid crystal properties of the system.

The focus of research turned to suspensions of large ferromagnetic particles in the nematic matrix. Brochard and de~Gennes identified the two key properties  of such systems:  strength of the mechanical coupling and stability of the suspension.  The former guarantees that the effect of the magnetic field on the liquid crystal, acted through the magnetic particles, is ability to control the nematic texture. The latter property sets a limit on the size and concentration of particles to prevent clustering. The numbers arrived at from theoretical considerations set the particle length  $l>0.5\times 10^{-2}\mu$m,  and a ratio $\frac{l}{d}\approx 10$, where $d$ denotes the diameter of the particle. The theoretical prediction on particle volume fraction was not to overcome the value $f= 10^{-3}$.  

In their experiments, Rault, Cladis and Burger \cite{RCB70} chose mono-domain particles of $\gamma{\textrm Fe}_2{\textrm O}\frac{2}{3}$, of $ 0.35 \mu\textrm{m}$ long (l) by $0.04$ $ \mu\textrm{m}$ in diameter (d). The saturation magnetization is $ 384 $ gauss with the easy axis parallel to the long axis of the grain.
Grains of these dimensions satisfy the criterion for mechanical coupling to the nematic liquid and also for mechanical rotation, as opposed to rotation of magnetization inside the grain, in a reversed field. Typical grain concentrations were of the order of $2\times 10^{11}$ grains/${\textrm cm}^3$, which corresponds to $f\approx 1.4\times 10^{-4}$, well within the theoretical prediction by Brochard and de Gennes.  For this physical parameters, Rault, Cladis and Burger state \cite{RCB70}: {\sl The ferronematic appeared to be very stable in the nematic-isotropic phases showing very little tendency to agglomerate. However, if a high field (1 kg) is applied to the sample in the isotropic phase, upon returning it to the nematic phase, we have observed long chains of grains about 50$ \mu\textrm{m}$.}  Both works assert that distortions of the nematic pattern in magnetic suspensions occur at very low fields: magnetizations range in the order of 0.1 to 1 gauss, instead of the values $ 10^{-4}$ to $10^{-3}$ of pure nematic liquid crystals, with a typical coupling gain of order $10^{3}$.

Central to the understanding of the nematic-magnetic coupling is the question of how the grains align in the nematic. Brochard and de~Gennes postulated strong anchoring of  nematic molecules along the magnetic moment, assumed to coincide with the direction of the particle axis. 
The effect of the grain magnetic field results from the anisotropy of the field around the grain, present even in the case of a spherical grain, resulting in a preferential direction for the magnetic moment in the nematic phase. This effect turns out to be small for small grains, with the magnetic moment causing only a local disruption of the nematic alignment.


In their experimental work, Chen and Amer [\cite{Chen-Amer83}, 1983] used particle coating that yields homeotropic anchoring of the liquid crystal on the magnetic grain to synthesize stable ferronematic systems. Although the length and aspect ratio of the particles, $0.5$ $\mu\textrm{m}$ and $7:1$, respectively, are compatible with those considered in the previous works, the earlier theory assuming rigid parallel anchoring was found to be not applicable to the homeotropic case.  The question of the surface anchoring and its implication on the relative orientation of $\bm$ and $\bn$ gave rise to an intense experimental and theoretical research activity spanning over three decades. In (\cite{raikher86}, \cite{burylov1-90}, \cite{burylov2-90}), the authors showed that the rigid anchoring approximation, $\bm||\bn$ might be used only if the condition $\frac{Wd}{K}>>1$ holds, where $W$ represents the surface energy density, and $K$ denotes a typical Frank constant.  A calculation for MBBA data, with $K= 5\times 10^{-7} \frac{\textrm{dyn}}{{\textrm cm}^2}, $ and $10^{-3}<W<10^{-2}$, and $d=0.07 \mu\textrm{m}$ gives $10^{-2}<\frac{Wd}{K}<10^{-1}$, showing a finite surface energy of the system.

Assuming soft liquid crystal surface anchoring, Burylov and Raikher (\cite{burylov-reiker95}, 1995) proposed a macroscopic free energy density of the form \begin{eqnarray}
  F=&& \frac{1}{2}\{K_1({ \textrm div}\,\bn)^2 + K_2({\textrm curl}\,\bn\cdot \bn)^2+ K_3(\bn\times{\textrm curl}\,\bn)^2\}-\frac{1}{2}\chi_a(\bn\cdot\bH)^2\nonumber\\
  && -M_s f\left(\bm\cdot\bH\right) +\left(\frac{f K_bT}{\nu}\right)\ln f+\left(\frac{A Wf}{d}\right)(\bn\cdot\bm)^2. \label{burylov-energy} \end{eqnarray}
Here $f$ represents the volume fraction of the particles, $\chi_a$ the anisotropic part of the diamagnetic susceptibility of nematic,  and the positive constants $\nu$ and  $M_s$ denote the particle volume and the saturation magnetization, respectively. In the last term,
$W$ represents the strength of the surface energy and $A=1-3\cos^2\alpha$ characterizes the type of anchoring, with  $\alpha$ denoting  the easy-angle orientation of the nematic on the particle surface. 

The macroscopic free energy (\ref{burylov-energy}) has been investigated in theoretical and experimental works involving orientational transitions in ferronematic states \cite{burylov2-90}, \cite{raiker98}, \cite{bena2003}. In particular, \cite{bena2003} presents a nonlinear modification of the Rapini-Papoular energy that predicts a first order Fredericks transition.  In {\cite {Kopcansky2001}} and \cite{Kopcansky99}, Kopcansky et al. use the modified theory to determine threshold fields in ferronematic transitions under combined electric and magnetic fields. In \cite{kopcansky2005}, the authors report on experimental studies of structural transitions in ferronematic subject to electric and magnetic field, with the matrix consisting of 8CB and 6CHBT liquid crystals, respectively. While in both cases the anchoring was determined as soft, it was found that $\bn\perp\bm $ in the first case, and $\bn\|\bm$ in the second. So, it was established then that both, parallel and perpendicular anchoring may occur depending on the properties of the matrix (which, in turn, reflects the properties of the particle coating).  Zadorozhnii et al. \cite{Z2006} provide a comprehensive analysis of the director---a unit vector in the direction of the preferred molecular alignment---switching for small and large values of the applied field in a nematic liquid crystal cell subject to homeotropic boundary conditions at the cell and particle walls. They show that the threshold field depends on the anchoring strength of the director on the particle surface. 

Note that a closely related set of models \cite{selinger1}-\cite{selinger2} exists for suspensions of ferroelectric nanoparticles in a nematic liquid crystalline matrix. The mechanical coupling between the particles and the nematic is still governed by the surface anchoring, but the particles interact with an electric and not a magnetic field.

In this work, we rigorously derive an expression for the effective ferronematic energy that reduces to the models described above under appropriate limits.   We consider a collection of spheroidal particles with, fixed, randomly distributed locations in the matrix, and with magnetic moment pointing in the direction of an easy axis.  The particles are taken as rotations and translations of the same spheroidal particle, located at the origin. We model the liquid crystalline matrix according to Ericksen's theory of nematics with variable degree of orientation. In this theory, the state of a liquid crystal is described by a vector $\bu(\bx)$ whose direction gives the average molecular alignment at the point $\bx$, while its magnitude $|\bu(\bx)|$ --- degree of orientation --- describes the quality of the alignment. Assuming that the Frank elastic constants are equal, the bulk liquid crystal energy has the form of the Ginzburg-Landau energy for $\bu$. We assume soft anchoring of the liquid crystal molecules on the surfaces of ferromagnetic particles as represented by the Rapini-Papoular energy term. The surface energy contribution can be either positive or negative depending on whether parallel or perpendicular alignment of nematic molecules is preferred on particles surfaces. It turns out that the case when the surface energy is negative is the most challenging to analyze.

Mathematically, we consider a family of energy functionals, $\mathcal F_\epsilon$, parametrized by a quantity $\epsilon>0$ that characterizes the geometry of the system, specifically, the size of the particles and the inter-particle distance. The system is assumed to be dilute, that is the volume fraction of the particles tends to $0$ in the limit $\epsilon\to 0$. The parameter scalings of the model that give the relative contribution of the different components of the energy are formulated in terms of $\epsilon$ as well. The choice of scalings guarantees that the limiting contributions of the bulk and surface energies, as well as the energy of interaction between the particles and the applied magnetic field are of order $O(1)$. We show that for the same parametric regime the contribution from the energy of magnetic interaction between the particles is $o(1)$ in $\epsilon$. This is consistent with the experimental observations that characterize dilute small particles systems in the absence of clustering.  

We study the variational limit of the family of energies $\{\mathcal F_{\epsilon}\}$ as $\epsilon\to 0$.  The limiting functional $\{\mathcal F_{0}\}$ represents the effective, or homogenized energy of the system.  Here the convergence is understood in the sense that the sequence of minimizers $\{u_{\epsilon}\}$ of $\{\mathcal F_{\epsilon}\}$ converges to a minimizer $u$ of $\{\mathcal F_{0}\}$ in an appropriate functional space.  The effective energy provides a benchmark for comparison with the formal expression for ferronematic energy functional \cite{burylov-reiker95} given in (\ref{burylov-energy}).

The homogenized energy (\ref{eq:limmagfun}) is more general than (\ref{burylov-energy}) as it is obtained under less restrictive assumptions. The interaction between the liquid crystal and the particles is due to surface anchoring and is represented by the matrix $A$ in (\ref{asstensor}) that encodes the information on  the shape and size of the particles, their locations, and their orientation with respect to a fixed frame.  Likewise, the effective magnetic moment $\bf M$ in (\ref{asstensor}) that couples the particles to the external magnetic field depends on the spatial and orientational distributions of the particles. For the high-aspect-ratio, needle-like particles the coupling terms reduce to their counterparts in (\ref{burylov-energy}).

\bigskip

\section{Background}
Given the domain $\Omega\subset{\mathbf{R}}^3$ let $P_i\subset\Omega$ be an arbitrary collection of subsets of $\Omega$ such that $P_i\cap P_j=\emptyset$ for every $i\neq j$ where $i,j=1,\ldots,n$. Suppose that the region $\Omega\backslash\cup_{i}P_i$ is occupied by a nematic liquid crystal and that for each $i=1,\ldots,n$ the region $P_i$ corresponds to a hard ferromagnetic particle embedded in the nematic matrix.

We will consider the liquid crystal configurations that can be described by the Ericksen's theory for nematics with variable degree of orientation; we will neglect all flow effects and assume that all elastic constants are equal. Further, we will use the phenomenological Rapini-Papoular term in order to approximate the liquid crystal/ferromagnetic surface energy. Then the elastic energy of the liquid crystal is given by
\[\mathcal{F}^{el}_{lc}:=\int_{\Omega\backslash\cup_{i}P_i}\left(K\,|\nabla{\bf{u}}|^2+W(|{\bf{u}}|)\right)\,dV+q\int_{\cup_{i}\partial
  P_i}{\left({\bf{u}},\nu\right)}^2\,d\sigma\,,\]
where $K>0$ is the elastic constant, $q\in\mathbf{R}$ is the strength of the surface term, $W$ is the bulk free energy of the undistorted state, and $\nu$ is the outward unit normal vector to $\partial P_i$.

Suppose that ferromagnetic particles are sufficiently small so that for every $i=1\,,\ldots\,N$ an $i-$th particle can be characterized by a magnetization vector ${\bf{m}}_i$ pointing in the direction of an easy axis of the particle. In order to derive the expression for the magnetostatic contribution $f^m$ to the free energy density of what is effectively a diamagnetic matrix interspersed with the ferromagnetic particles, we follow \cite{Landau-Lifshitz-ECM}. We have that
\begin{equation}
  \label{mfe1}
  \left(\frac{\partial f^m}{\partial {\bf H}}\right)_T=-\bf{B}\,,
\end{equation}
where ${\bf H}$ and ${\bf B}$ are the magnetic field and the magnetic induction, respectively (cf. eq. (39.1) in \cite{Landau-Lifshitz-ECM}) and the derivative is taken holding the temperature $T$ fixed. Assuming that ${\bf M}$ denotes the magnetic moment, the induction is given by \begin{equation}
  \label{mfe2}
  {\bf B}=\mu_0({\bf H}+{\bf M})\,,
\end{equation}
where $\mu_0$ is the magnetic permeability of vacuum.

Suppose that the magnetic moment of the material can be written as
\begin{equation}
  \label{mfe3}
  {\bf M}={\bf m}+\chi{\bf H}\,,
\end{equation}
and the material can exhibit both the spontaneous magnetization ${\bf
  m}$ (an independent thermodynamic variable) and the magnetization induced by the field (we assume it to be proportional to the field).  The tensor $\chi$ is the magnetic susceptibility; it is generally small in diamagnetics, but it can be large in soft ferromagnetic bodies. In what follows, we will set ${\bf m}={\bf 0}$ in the liquid crystal while we will set $\chi=0$ in hard ferromagnetics. 

Substituting (\ref{mfe2}) and (\ref{mfe3}) into (\ref{mfe1}) and integrating with respect to the field, we obtain
\begin{equation}
  \label{mfe4}
  f^m({\bf m},{\bf H})=f^m({\bf m},{\bf 0})-\mu_0({\bf m},{\bf H})-\frac{\mu({\bf H},{\bf H})}{2}
\end{equation}
Here $\mu=\mu_0({\bf I}+\chi)$ is the magnetic permeability tensor.  Note that the energy $f^m({\bf m},{\bf 0})$ accounts for both the exchange and anisotropy energies for a ferromagnetic body. We will ignore this splitting since we consider single-domain particles.

The expression (\ref{mfe4}) can be adjusted further by excluding the energy of the external field that would otherwise be created by the same sources in vacuum.

Let the fields ${\bf H}$ and ${\bf h}$ solve the (different) sets of Maxwell's equations under the same boundary conditions at infinity in the presence and in the absence of the material, respectively. Then ${\bf h}$ is the magnetic field in vacuum when there is no magnetizing body (cf. eq. (32.1) in \cite{Landau-Lifshitz-ECM}).

Since the free energy of the field ${\bf h}$ is
\[\mathcal{F}^m_{\bf h}:=-\int_{\mathbf{R}^3}\frac{\mu_0{|{\bf
      h}|}^2}{2}\,dV\,,\]
the adjusted free energy can be written as
\begin{equation}
  \label{mfe5}
  \tilde {\mathcal{F}}^m:=\int_{\mathbf{R}^3}f^m\,dV-\mathcal{F}^m_{\bf h}=\int_{\mathbf{R}^3}\left(f^m+\frac{\mu_0{|{\bf h}|}^2}{2}\right)\,dV\,.
\end{equation}
By rearranging terms, using Maxwell's equations, and integrating,
one can show \cite{Landau-Lifshitz-ECM} that
\begin{equation}
  \label{mfe6}
  \tilde {\mathcal{F}}^m=\int_{\mathbf{R}^3}\left(f^m+\frac12({\bf H},{\bf B})-\frac{\mu_0}{2}({\bf M},{\bf h})\right)\,dV\,.
\end{equation}
This equation can be simplified by taking (\ref{mfe4}) into account to obtain
\begin{equation}
  \label{mfe7}
  \mathcal{F}^m=-\frac{\mu_0}{2}\int_{\mathbf{R}^3}\left(({\bf m},{\bf H})+({\bf m},{\bf h})+\chi({\bf H},{\bf h})\right)\,dV\,,
\end{equation}
where we dropped the tilde for convenience.

In a hard ferromagnetic material, the magnetic susceptibility
$\chi=0$. By denoting the demagnetizing field by ${\bf H}_i={\bf
  H}-${\bf h} the equation (\ref{mfe7}) reduces to
\[\mathcal{F}^m=-\frac{\mu_0}{2}\int_{\mathbf{R}^3}\left(({\bf
    m},{\bf H}_i)+2({\bf m},{\bf h})\right)\,dV\,,\]
---this is the sum of the magnetostatic and the Zeeman energies.  Further, ${\bf H}_i$ vanishes as $x\to\infty$ and it satisfies the same set of Maxwell's equations as ${\bf H}$.

If the material is diamagnetic, then ${\bf m}={\bf 0}$ and $\chi$ is small enough so that the magnetic field is essentially unperturbed by the presence of magnetizing body. We conclude that
\[\mathcal{F}^m=-\frac{\mu_0}{2}\int_{\mathbf{R}^3}\chi({\bf
  h},{\bf h})\,dV\,,\]
which is the standard form of the free energy for the diamagnetic bodies.

Now we establish the expressions for the magnetic free energy in various components of the composite. Suppose for now that the external field ${\bf h}$ is constant.

Using the same notation as above, the energy of interaction between the magnetic field and the (diamagnetic) liquid crystal (cf.  \cite{virga}, \cite{mottram}) is given by
\[\mathcal{F}^{m}_{lc}:=-\frac{\mu_0}{2}\int_{\Omega\backslash\cup_{i}P_i}\chi_{lc}
({\bf H},{\bf h})\,dV\,.\]
The magnetic susceptibility tensor $\chi_{lc}$ can be approximated as
\[\chi_{lc}=\frac{\chi_a|\bf{u}|}{s_{exp}}\left(\frac{\bf{u}}{|\bf{u}|}\otimes\frac{\bf{u}}{|\bf{u}|}-\frac{1}{3}{\bf{I}}\right)+\bar\chi\,{\bf{I}}\,.\]
Here $\chi_a=\chi_\parallel-\chi_\perp$ is the rescaled diamagnetic anisotropy and $$\bar\chi=\left(\chi_\parallel+2\chi_\perp\right)/3,$$ is the average susceptibility. The scaling factor $s_{exp}$ is the value of the uniaxial order parameter $|\bf{u}|$ when the measurements of the susceptibility were taken, and it reflects the hysteresis behavior of the magnetic loading experiments.  We point out that in a nematic $\chi_\parallel\,,\chi_\perp<0$ and $0<\chi_a\ll|\bar\chi|$ \cite{degennes}. The smallness of $\chi_a/\bar\chi$ is the basis for assuming that the effect of the liquid crystal on the magnetic field is weak \cite{ virga}.

By setting $\chi=0$ in (\ref{mfe7}), the free energy of the hard ferromagnetic particles is
\[\mathcal{F}^m_f:=-\frac{\mu_0}{2}\sum_{i=1}^{N}\int_{P_i}\left\{({\bf
    m}_i,{\bf H})+({\bf m}_i,{\bf h})\right\}\,dV\,.\]
By solving the Maxwell's equations of magnetostatics, we find that the total field ${\bf H}$ is given by
\[{\bf H}=-\nabla\phi\,,\]
where the magnetic potential satisfies the  equations
\begin{equation*}
\begin{split}
\Delta\phi&=0, \quad \textrm{ in }\, \cup P_i,\\
\mathrm{div}\left(\mu_{lc}\nabla\phi\right)&=0, \quad \textrm{in }\, {\left\{\cup P_i\right\}}^c.
\end{split}
\end{equation*}
 The boundary conditions are
\[\left.\left[-\mu\frac{\partial\phi}{\partial\nu}+({\bf{m}}_i,\nu)\right]\right|_{\partial
  P_i}=0\,,\] for every $i=1,\ldots,N$ and
\[\nabla\phi={\bf h}\,,\]\,
at infinity. Here the magnetic permeability tensor $\mu=\mu_{lc}=\mu_0\left({\bf I}+\chi_{lc}\right)$ in the liquid crystal and $\mu=\mu_0{\bf I}$ in the ferromagnetic particles.

The equilibrium configuration of the composite can be found by
minimizing the functional
\[\mathcal{F}:=\mathcal{F}_{lc}^{el}+\mathcal{F}_{lc}^m+\mathcal{F}^m_f\,,\]
with respect to ${\bf{u}}$ and ${\bf{m}}_i$.

\section {Formulation of the problem}
Suppose that the positions and orientations of prolate spheroidal particles are fixed and distributed randomly in the matrix, the spontaneous magnetic moments of the ferromagnetic particles are parallel to their long axes, and $\chi_{a}=0$.

Consider the family of energy functionals ${\mathcal F}_{\ep}$ 
\begin{equation}
  \label{energy}
  \begin{split}
    \mathcal F_\ep[\bu]=&\int_{\Omega\setminus\cup{\calP}_i^\ep}\big\{|\nabla \bu|^2+ W(|\bu|)\big\}\,dV + g_{\ep}\int_{\cup\partial {\calP_i^\ep}}(\bu,\mathbf\nu)^2\,d\sigma \\
    -&\int_{{\mathbf R}^3}\left\{(\bm_\ep,\mathbf H_\ep) +
      2(\bm_\ep,\mathbf h_\ep)\right\}\,dV,
  \end{split} \end{equation} 
where $\ep>0$ is a small parameter related to the geometry of the system. Here \begin{equation}
  \left\{
    \begin{array}{ll}
      \bm_\ep=\bm^\ep_i, &\bx\in\calP_i^\ep \\
      0, &\bx\in\Omega\setminus\cup\calP_i^\ep.
    \end{array}
  \right.
\end{equation}
and for simplicity, we set $W(t)={\left(1-t^2\right)}^2$. The magnetic
field is given by
\begin{equation}
  h_{\ep}=\left|{\bf h}_\ep\right|=\textrm {constant},\quad  \bH_\ep=-\nabla\varphi,
\end{equation}
with
\begin{equation}
\label{mag_energy}
  \left\{
    \begin{array}{ll}
      \triangle\varphi=0, & x\in{\mathbf R}^3,\\
      \left.\left[-\mu_0\frac{\partial\varphi}{\partial\nu}+(\bm_i^\ep,\nu)\right]\right|_{\partial\calP_i}=0, & x\in\partial\calP_i^\ep.
    \end{array}
  \right.
\end{equation}
We assume that for a prescribed $\mathbf U\in C^1(\Omega, {\mathbf
  R}^3)$, \begin{equation}{\buep}=\mathbf U, \quad \textrm {on}\,\,
  \partial\Omega.\end{equation} 
For each $\ep>0$, we denote by $\buep$ the minimizer of (\ref{energy}).  We study the limiting energy and the behavior of minimizers of ${\mathcal
  F}_{\ep}$ as $\ep\to 0$.

We make the following assumptions:
\begin{enumerate}
\item The ferromagnetic particles consist of a family of $N_\ep$ prolate
  spheroids $\calP_i^\ep=\bx_i^\ep+ \ep^\alpha R_i^\ep\calP, \, i=1,
  ...N_\ep$, where $\bx_i^\ep\in{\mathbf R}^3$ denotes a particle
  center and $\calP$ is a reference spheroid with the long axis
  parallel to the $z$-coordinate axis, and $R_i^\ep$ is a rotation.
\item Given positive numbers $0<d<D$, the distance between particles
  $|\bx_i^\ep-\bx_j^\ep|\in[d\ep, D\ep]$, for all, $0<i,j,\leq N_\ep$.
  Thus $N_\ep<N\ep^{-3}$ for some $N>0$ uniformly in $\ep$.
\item
  $\left|\bm_i^{\ep}\right|=m_\ep={\mathrm{Vol}}(\calP_i^\ep)m\ep^{\beta_1}$,
  ${\bf h}_\ep={\bf h}\ep^{\beta_2}$, and $ g_{\ep}=g\ep^{\gamma}, $ where $m$,
  ${\bf h}$, and $g$ are given constants.
\item The parameters $\alpha, \beta_1, \beta_2,$ and $\gamma$ satisfy
  \begin{equation}
    \label{asspars}
    \begin{split}
    1<\alpha<2, &\quad 6\alpha+2\beta_1>9,\\ \beta_2+\beta_1=3-6\alpha, &\quad \gamma=3-2\alpha.
  \end{split}
  \end{equation}
\item The matrix-valued functions
\begin{equation}
\label{asstensor}
\begin{split}
  A^\ep\left(\bx\right)&=\ep^3g\sum_i\delta\left(\bx-\bx_i^\ep\right)R_i^\ep\left(\int_{\partial\calP}\nu\otimes\nu\,d\sigma\right) {R_i^\ep}^T, \\
  &{\bf M}^\ep\left(\bx\right)={\ep^3m\mathrm
    {Vol}}^2(\calP)\sum_i\delta\left(\bx-\bx_i^\ep\right)R_i^\ep \hat{\bf z},
\end{split}
\end{equation} 
converge in the sense of distributions to $A,{\bf M}\in L^\infty(\Omega)$, respectively, where $A:\mathbb R^3\to M^{3\times 3}$ and ${\bf M}:\mathbb R^3\to \mathbb R^3$. Here $\hat{\bf z}$ is a unit vector in the direction of $z$-axis.
\end{enumerate}
\noindent {\bf Remark.\,} Note that the total volume of the particles
satisfies $\mathrm{Vol}\left(\cup
  \calP_i^\ep\right)=O\left(\ep^{3(\alpha-1)}\right)$, so that the homogenization problem for (\ref{energy}) corresponds to a dilute limit when $\lim \mathrm{Vol}\left(\cup\calP_i^\ep\right)\longrightarrow 0$ as $\ep\to 0$. The scalings on $\xi_{\ep}$ and $g_{\ep}$ guarantee that the magnetic interaction between the applied field and the particles, the Ginzburg-Landau energy, and the surface energy are all $O(1)$ while the magnetic interactions between the particles are of order $o(1)$ and, therefore, can be neglected.

Our principal goal is to prove the following
\begin{theorem}
  \label{t_main}
Suppose that the assumptions 1-5 hold. Then the sequence of minimizers $\left\{\bu_\epsilon\right\}_{\epsilon>0}$ of the functionals $\mathcal F_\epsilon[\bu_\epsilon]$ converges in the sense of (\ref{convergence}) to a minimizer of the functional
\begin{equation}
  \label{eq:limmagfun}
  \mathcal F _0[\bu]=\int_{\Omega}\left[{\left|\nabla\bu\right|}^2+{\left(1-\left|\bu\right|^2\right)}^2+\left(A\bu,\bu\right)-2({\bf h},{\bf M})\right]\,dV,
\end{equation}
where $A$ and $\bf M$ are as defined in assumption 5.
\end{theorem}
The matrix $A$ and the vector $\bf M$ that appear in the statement of Theorem \ref{t_main} describe the homogenized liquid crystal/ferromagnetic particles interaction and the effective magnetization density, respectively.

\section{Main Results}
We prove Theorem \ref{t_main} in several steps as outlined below.

\subsection{Liquid Crystal Energy} First, we consider the energy (\ref{energy}) without the magnetic terms, that is 
\begin{equation}
  \label{lc-energy}
  \mathcal E_\ep[\bu]= \int_{\Omega\setminus\cup{\calP}_i^{\ep}}\big\{|\nabla \bu|^2+ W(|\bu|)\big\}\,dV + g_{\ep}\int_{\cup\partial {\calP_i^{\ep}}}(\bu,\mathbf\nu)^2\,d\sigma.
\end{equation}
For each small $\ep>0$, we let $\buep$ be the minimizer of
(\ref{lc-energy}) subject to the Dirichlet boundary condition $\buep=
\bU$ on $\partial \Omega$.

We want to find the limiting functional of the family $\mathcal E_{\ep}$ as $\epsilon\to 0$. Although our approach is developed for the prolate spheroidal particles, it can be easily extended to particles of arbitrary convex shapes. The method is based on the procedure developed in \cite{BerKhrus} for the case of spheres.

\subsubsection{ Compactness} We first observe that the restriction of $\mathbf U$ to the domain $\Omega_{\ep}= \Omega\setminus\cup \calP_i^{\ep} $ is an admissible function. Indeed, \begin{equation}
  \begin{split}
    \mathcal E_\ep[\bU]=&\int_{\Omega\setminus\cup{\calP}_i^{\ep}}\big\{|\nabla \bU|^2+ W(|\bU|)\big\}\,dV + g_{\ep}\int_{\cup\partial {\calP_i^{\ep}}}(\bU,\mathbf\nu)\,d\sigma\\ \leq &\int_{\Omega}\big\{|\nabla \bU|^2+ W(|\bU|)\big\}\,dV + g_{\ep}\int_{\cup\partial {\calP_i^{\ep}}}(\bU,\mathbf\nu)\,d\sigma\\
    \leq & C\left(1 + g_{\ep}
      N_{\ep}|\partial\calP|\ep^{2\alpha}\right)\leq
    C\left(1+gN|\partial\calP|\right)\leq C,
  \end{split}
\end{equation}
where $C$ is a generic positive constant.  Consequently,
\begin{equation}
  \label {uniform-boundedness}
  \mathcal E_\ep[\buep]\leq\mathcal E_\ep[\bU]\leq C. 
\end{equation}
That is, $\mathcal E_\ep[\buep]$ is uniformly bounded in $\ep$.

The following lemma is needed towards the proof of compactness of the
sequence $\{\buep\}$ of energy minimizers of (\ref{lc-energy}).
\begin{lemma}
  \label{l1}
  Let $\calP$ denote a prolate spheroid in $\mathbf R^3$ with minor
  and major axes $A$ and $B$, respectively. Let
  $\hat\calP\supset\calP$ represent the prolate spheroid homothetic to
  $\calP$ with axes $\frac{\hat A}{A}=\frac{\hat B}{B}>2$.  Then
  \begin{equation}
    \label{lemma1}
    \begin{split}
      \int_{\partial \calP} |\bu|^2\,d\sigma & \leq
      \frac{3B^2(1+\lambda)}{A}\int_{\hat\calP\setminus\calP}|\nabla
      \bu|^2\,dV \\ &+
      \left(1+\frac{1}{\lambda}\right)\frac{24A^2}{{7\hat A}^3}
      \int_{\hat\calP\setminus\calP}|\bu|^2\,dV.
    \end{split}
  \end{equation}  
\end{lemma}
\begin{proof}
  Suppose that the center of the spheroid $\calP$ is at the origin and
  its long axis is oriented along $z$-axis. We introduce the
  coordinates
  \[x=\rho\sin{\phi}\cos{\theta},\ y=\rho\sin{\phi}\sin{\theta},\
  z={A}^{-1}{B}\rho\cos{\phi},\] then the volume element is given by
  $dV={A}^{-1}{B}\rho^2\sin{\phi}\,d\rho\,d\theta\,d\phi$ and $\rho=C$
  defines a prolate spheroid with axes $C$ and $BC/A$ with the surface
  area element
  $d\sigma=A^{-1}C^2\sin{\phi}\sqrt{B^2\sin^2{\phi}+A^2\cos^2{\phi}}\,d\theta\,d\phi$.
  We start with the relation
  \begin{equation}
    \label {e1}
    \bu(A,\phi, \theta)= \bu(t, \phi,\theta)-\int_A^t \bu_{\rho}d\,\rho, \quad \textrm {where}\,\, t\in[A,\hat A].
  \end{equation}
  Let $\lambda>0$ be fixed. Taking the square of (\ref{e1}) and
  applying Young's inequality gives
  \begin{equation}
    \label{l1.1}
    \begin{split}
      {|\bu|}^2(A,\phi,\theta)&={|\bu|}^2(t,\phi,\theta)-2\bu(t,\phi,\theta)\cdot\int_A^{t} \bu_{\rho}d\rho+\left|\int_A^t \bu_{\rho}d\rho\right|^2 \\
      &\leq(1+\lambda)\left|\int_A^t \bu_{\rho}d\rho\right|^2+ \left(1+\frac{1}{\lambda}\right){|\bu|}^2(t,\phi,\theta).
    \end{split}
  \end{equation}
  Further, by H\"older's inequality
  \[
  \begin{split}
    \left|\int_A^{t}\bu_{\rho}d\,\rho\right|^2&\leq\int_A^t{\left|\bu_\rho\right|}^2\rho^2d\rho\,\int_A^t\rho^{-2}d\rho \\
    &\leq\frac{1}{A}\int_A^{\hat
      A}{\left|\bu\right|}_\rho^2\rho^2d\rho
  \end{split}
  \]
  We multiply both sides of the inequality (\ref{l1.1}) by the
  determinant of the Jacobian, integrate in $\hat\calP\setminus\calP$,
  and use the fact that
  ${\left|\bu_\rho\right|}^2\leq3B^2A^{-2}{|\nabla \bu|}^2$:
  \begin{equation}
    \label{eq2}
    \begin{split}
      \frac{B}{A}&\int_0^{\pi}\int_0^{2\pi}\int_A^{\hat A}{|\bu|}^2(A,\phi,\theta)\rho^2\sin{\phi}\,d\rho\,d\theta\,d\phi \\
      &\leq\frac{1+\lambda}{A} \frac{{\hat A}^3-A^3}{3}\int_0^{\pi}\int_0^{2\pi}\int_A^{\hat A}{|\bu_\rho|}^2A^{-1}B\rho^2\sin{\phi}\,d\rho\,d\theta\,d\phi \\
      &+\left(1+\frac{1}{\lambda}\right) \int_0^{\pi}\int_0^{2\pi}\int_A^{\hat A}{|\bu|}^2A^{-1}B\rho^2\sin{\phi}\,d\rho\,d\theta\,d\phi \\
      &\leq(1+\lambda)\frac{B^2\left({\hat
            A}^3-A^3\right)}{A^3}\int_{\hat\calP\setminus\calP}{|\nabla
        \bu|}^2\,dV+\left(1+\frac{1}{\lambda}\right)
      \int_{\hat\calP\setminus\calP}{|\bu|}^2\,dV.
    \end{split}
  \end{equation}
  At the same time
  \begin{equation}
    \label{eq3}
    \begin{split}
      \frac{B}{A}&\int_0^{\pi}\int_0^{2\pi}\int_A^{\hat A}{|\bu|}^2(A,\phi,\theta)\rho^2\sin{\phi}\,d\rho\,d\theta\,d\phi \\
      &=\frac{B({\hat A}^3-A^3)}{3A}\int_0^{\pi}\int_0^{2\pi}{|\bu|}^2(A,\phi,\theta)\sin{\phi}\,d\theta\,d\phi \\
      &\geq\frac{{\hat A}^3-A^3}{3A^2}\int_0^{\pi}\int_0^{2\pi}{|\bu|}^2(A,\phi,\theta)a\sin{\phi}\sqrt{B^2\sin^2{\phi}+A^2\cos^2{\phi}}\,d\theta\,d\phi \\
      &=\frac{{\hat
          A}^3-A^3}{3A^2}\int_{\partial\calP}{|\bu|}^2\,d\sigma\,.
    \end{split}
  \end{equation}
  Combining (\ref{eq2}) with (\ref{eq3}) and using the fact that $\hat
  A>2A$, we obtain
  \[
  \begin{split}
    \int_{\partial\calP}&{|\bu|}^2\,d\sigma\leq \frac {3B^2 (1+\lambda)}{A}\int_{\hat\calP\setminus\calP}{|\nabla \bu|}^2\,dV+\left(1+\frac{1}{\lambda}\right) \frac{3A^2}{{\hat A}^3-A^3}\int_{\hat\calP\setminus\calP}{|\bu|}^2\,dV \\
    &\leq \frac {3B^2
      (1+\lambda)}{A}\int_{\hat\calP\setminus\calP}{|\nabla
      \bu|}^2\,dV+\left(1+\frac{1}{\lambda}\right) \frac{24A^2}{7{\hat
        A}^3}\int_{\hat\calP\setminus\calP}{|\bu|}^2\,dV\,.
  \end{split}
  \]
\end{proof}

Next, we use the previous lemma to estimate the surface energy
contribution in (\ref{lc-energy}) in terms of the $L^2-$norms of $\bu$
and $\nabla\bu$.
\begin{lemma}
  \label{l2}
  Let $\ep>0$, $\lambda>0$ be as in Lemma \ref{l1}. Then
  \begin{equation}
    \label{surface-estimate}
    \begin{split}
      g_{\ep}&\int_{\cup\partial\calP_i^\ep}(\bu\cdot\nu)^2\,d\sigma\leq C(1+\lambda)\left[\ep\int_{\omC}|\nabla\bu|^2\,dV\right. \\
      &\left.+\lambda^{-1} \int_{\omC}|\bu|^2\,dV\right]\,,
    \end{split}
  \end{equation}
  for any admissible function $\bu$, where the constant $C$ is independent of $\ep$.
\end{lemma}
\begin{proof}
  Let $C$ denote a generic constant independent of $\ep$. Setting
  $A=\ep^\alpha a$, $B=\ep^\alpha b$, and $\hat A=d\ep/2$, we apply
  Lemma \ref{l1} to the surface integral term
  \[
  \begin{split}
    \int_{\partial\calP_i^\ep}&(\bu\cdot\nu)^2\,d\sigma\leq\int_{\partial\calP_i^\ep}|\bu|^2\,d\sigma\\
    &\leq \ep^\alpha\frac {3b^2 (1+\lambda)}{a}\int_{\hat\calP_i^\ep\setminus\calP_i^\ep}{|\nabla \bu|}^2\,dV+\ep^{2\alpha-3}\left(1+\frac{1}{\lambda}\right) \frac{192a^2}{7d^3}\int_{\hat\calP_i^\ep\setminus\calP_i^\ep}{|\bu|}^2\,dV \\
    &\leq
    C(1+\lambda)\left[\ep^\alpha\int_{\hat\calP_i^\ep\setminus\calP_i^\ep}{|\nabla
        \bu|}^2\,dV+\ep^{2\alpha-3}\lambda^{-1}\int_{\hat\calP_i^\ep\setminus\calP_i^\ep}{|\bu|}^2\,dV\right].
  \end{split}
  \]
  Then, since $g_\ep=g\ep^{3-2\alpha}$ and $1<\alpha<2$, we have
  \begin{equation}
    \label{bdrest3}
    \begin{split}
      g_\ep&\int_{\cup \partial P_i^\ep}(\bu\cdot\nu)^2\,d\sigma=g_\ep\sum_{i=1}^{N_\ep}\int_{\partial\calP_i^\ep}(\bu\cdot\nu)^2\,d\sigma \\
      &\leq C(1+\lambda)\sum_{i=1}^{N_\ep}\left[\ep\int_{\hat\calP_i^\ep\setminus\calP_i^\ep}{|\nabla \bu|}^2\,dV+\lambda^{-1}\int_{\hat\calP_i^\ep\setminus\calP_i^\ep}{|\bu|}^2\,dV\right] \\
      &\leq C(1+\lambda)\left[\ep\int_{\omC}{|\nabla
          \bu|}^2\,dV+\lambda^{-1}\int_{\omC}{|\bu|}^2\,dV\right]\,.
    \end{split}
  \end{equation}
\end{proof}
We are now in the position to prove the following theorem
\begin{theorem}
  \label{t1}
  If a sequence of admissible functions $\left\{\buep\right\}$
  satisfies $\mathcal E_\ep\left[\buep\right]<M$ for some constant
  $M>0$ uniformly in $\ep$, then there exists a constant $\tilde M>0$
  such that
  $\left\|\buep\right\|_{H^1\left(\Omega\backslash\cup_i\calP_i^\ep\right)}<\tilde
  M$ uniformly in $\ep$.
\end{theorem}
\begin{proof}
  Suppose that $\left\{\buep\right\}$ satisfies $\mathcal
  E_\ep\left[\buep\right]<M$ for some constant $M>0$ uniformly in
  $\ep$. Using Lemma \ref{l2} with $\lambda=1$, the assumption on $W(t)$, and
  H\"older's inequality, we have
  \[
  \begin{split}
    \int_{\omC}&\left\{{|\nabla\bu|}^2+{|\bu|}^4\right\}\,dV\leq M+2\int_{\omC}{|\bu|}^2\,dV+\left|g_\ep\right|\int_{\cup_i\partial\calP_i^\ep}{(\bu\cdot\nu)}^2\,d\sigma \\
    &-\left|\omC\right|\leq C\epsilon\int_{\omC}{|\nabla\bu|}^2\,dV \\
    &+C\int_{\omC}{|\bu|}^2\,dV+M_1 \\
    &\leq C\epsilon\int_{\omC}{|\nabla\bu|}^2\,dV \\
    &+C{|\Omega|}^{1/2}{\left(\int_{\omC}{|\bu|}^4\,dV\right)}^{\frac{1}{2}}+M_1\,,
  \end{split}
  \]
  where $M_1>0$ is a constant independent of $\ep$. Let $\ep$ be small
  enough so that $C\epsilon<\frac{1}{2}$. Then
  \[
  \int_{\omC}\left(\frac{1}{2}{|\nabla\bu|}^2+{|\bu|}^4\right)\,dV\leq
  M_2\left[1+{\left(\int_{\omC}{|\bu|}^4\,dV\right)}^{\frac{1}{2}}\right]\,,
  \]
  uniformly in $\ep$ for some constant $M_2>0$. Using the same
  arguments as in \cite{BerKhrus} we conclude that there exists a
  constant $\tilde M>0$ such that
  $\left\|\buep\right\|_{H^1\left(\Omega\backslash\cup_i\calP_i^\ep\right)}<\tilde
  M$ uniformly in $\ep$.
\end{proof}
\begin{remark}
  Note that the proof of Theorem \ref{t1} is trivial if $g>0$ when the
  boundary term is nonnegative.
\end{remark}
Due to our assumptions on the distributions and the sizes of the
spheroids $\calP_i^\ep$, the domains in the sequence $\omC$ are
strongly connected \cite{Khruslov78} that is, for every function
$\bu\in H^1\left(\omC\right)$, there exists an extension $\tilde\bu\in
H^1(\Omega)$ such that
\begin{equation}
  \label{eqext}
  \left\|\tilde\bu\right\|_{H^1(\Omega)}\leq C\left\|\bu\right\|_{H^1\left(\omC\right)}
\end{equation}
where $C>0$ is independent of $\ep$. Note that a sufficient condition
for (\ref{eqext}) is the existence of a "security layer" around each
particle having thickness comparable with the diameter of the
particle as $\ep\to 0$ \cite{Berlyand83}. It follows that there
exists a sequence $\{\tilde\bu_\ep\}$ of extended minimizers that is
uniformly bounded in $H^1(\Omega)$ and, up to a subsequence, converges
to some $\bu_0$ weakly in $H^1(\Omega)$ and strongly in $L^2(\Omega)$.
Thus
\begin{equation}
  \label{convergence}
  \int_{\omC}{\left|\bu_\ep-\bu_0\right|}^2\,dV\to 0\,,
\end{equation}
as $\ep\to 0$. Further, by trace theorem,
\[
(\bu_0-\bU)|_{\partial\Omega}={\bf 0}.
\]

In order to identify the limiting functional and to demonstrate that
$\bu_0$ is its minimizer, we now prove
\begin{theorem}
Suppose that
\begin{equation}
  \label{limfun}
  \mathcal E[\bu]:=\int_{\Omega}\left[{\left|\nabla\bu\right|}^2+{\left(1-\left|\bu\right|^2\right)}^2+\left(A\bu,\bu\right)\right]\,dV\,,
\end{equation}
for every $\bu\in H^1(\Omega)$. Given $\bw\in C^\infty(\bar\Omega)$, there exists a sequence $\left\{\bw^\ep\right\}\subset H^1(\Omega)$ such that
\begin{equation}
  \label{realseq}
  \mathcal E_\ep\left[\bw^\ep\right]\to\mathcal E[\bw]\,,
\end{equation}
when $\ep\to 0$.
\end{theorem}
\begin{proof}
  We begin by constructing a test function. Let $\bw\in
  C^\infty(\bar\Omega)$ and set
  \begin{equation}
    \label{testfn}
    \bw^\ep:=\bw+\bz^\ep=\bw+\sum_i\left(\bu_i^\ep-\bw\right)\phi\left(\frac{\left|\bx-\bx_i^\ep\right|}{\ep^\kappa}\right)\,,
  \end{equation}
  where $\kappa\in(1,\alpha)$, the function $\phi\in
  C^\infty(\mathbb{R}^+)$ satisfies
  \[
  \phi(t)=\left\{
    \begin{array}{ll}
      1, & \mathrm{if}\ t<\frac{1}{2}\,, \\
      0, & \mathrm{if}\ t>1\,,
    \end{array}
  \right.
  \]
  For every $i=1,\ldots,N_\ep$, the function $\bu_i^\ep$ is a solution
  of the following problem
  \begin{equation}
    \label{ui}
    \left\{
      \begin{array}{ll}
        \Delta \bu_i^\ep-\frac{1}{\ep^{2\alpha}}\left(\bu_i^\ep-\bw_i\right)={\bf 0}, & \mathrm{in}\ B_{\ep^\kappa}(\bx_i^\ep)\backslash\calP_i^{\ep}\,, \\
        \frac{\partial\bu_i^\ep}{\partial\nu}+g_\ep\left(\bu_i^\ep,\nu\right)\nu={\bf 0}, & \mathrm{on}\ \partial\calP_i^{\ep}\,, \\
        \bu_i^\ep=\bw_i, & \mathrm{when}\ |\bx|=\ep^\kappa\,,
      \end{array}
    \right.
  \end{equation}
  where $\bw_i=\bw(\bx_i^\ep)$.

  To understand the behavior of a solution to (\ref{ui}), for a fixed
  $i\in\{1,\ldots,N_\ep\}$, we rescale the lengths by the
  characteristic size of the particle:
  $\by=\ep^{-\alpha}(\bx-\bx_i^\ep)$ and set $\hat
  \bu_i^\ep(\by):=\bu_i^\ep\left(\bx_i^\ep+\ep^\alpha\by\right)-\bw_i$.
  Then
  \begin{equation}
    \label{uind}
    \left\{
      \begin{array}{ll}
        \Delta \hat\bu_i^\ep-\hat\bu_i^\ep={\bf 0}, & \mathrm{in}\ B_{\ep^{\kappa-\alpha}}({\bf 0})\backslash\calP_i\,, \\
        \frac{\partial\hat\bu_i^\ep}{\partial\nu}+g\ep^{3-\alpha}\left(\hat\bu_i^\ep+\bw_i,\nu\right)\nu={\bf 0}, & \mathrm{on}\ \partial\calP_i\,, \\
        \hat\bu_i^\ep={\bf 0}, & \mathrm{when}\ |\by|=\ep^{\kappa-\alpha}\,,
      \end{array}
    \right.
  \end{equation}
  where the spheroid $\calP_i=\ep^{-\alpha}\calP^i_\ep$ is centered at
  the origin.  Note that $\hat\bu_i^\ep$ is a critical point of the
  functional
  \begin{equation}
    \label{local_func}
    \hat E^i_\ep[\bu]:=\int_{B_R({\bf 0})\backslash\calP_i}\left[{|\nabla\bu|}^2+{\left|\bu\right|}^2\right]\,dV+g\ep^{3-\alpha}\int_{\partial\calP_i}{(\bu+\bw_i,\nu)}^2\,d\sigma\,,
  \end{equation}
  where $\bu\in H^1_0\left(B_R({\bf 0})\backslash\calP_i\right)$ and
  $R=\ep^{\kappa-\alpha}$. We can assume that $\hat\bu_i^\ep$ is a
  global minimizer of $\hat E^i_\ep$ over $H^1_0\left(B_R({\bf
      0})\backslash\calP_i\right)$ once we prove the following
  \begin{lemma}
    The $\min_{H^1_0\left(B_R({\bf 0})\backslash\calP_i\right)}\hat
    E^i_\ep$ is attained and the minimizer satisfies
    \begin{align}
      \label{l31}
      &\int_{B_R({\bf 0})\backslash\calP_i}{\left|\nabla\hat\bu_i^{\ep}\right|}^2\,dV\leq C\ep^{6-2\alpha}\,,\\
      \label{l32}
      &\int_{B_R({\bf 0})\backslash\calP_i}{\left|\hat\bu_i^{\ep}\right|}^2\,dV\leq C\ep^{6-2\alpha}\,,\\
      \label{l33}
      &\int_{\partial\calP_i}{\left|\hat\bu_i^{\ep}\right|}^2\,d\sigma\leq
      C\ep^{6-2\alpha}\,.
    \end{align}
  \end{lemma}
  \begin{proof}
    1. {\em Boundedness from above}. Since $\bu\equiv{\bf 0}$ is in
    $H^1_0\left(B_R({\bf 0})\backslash\calP_i\right)$,
    \begin{equation}
      \label{l3e0}
      \min_{H^1_0\left(B_R({\bf 0})\backslash\calP_i\right)}\hat E^i_\ep \leq \hat E^i_\ep\left[{\bf 0}\right]=g\ep^{3-\alpha}\int_{\partial\calP_i}{\left(\bw(\bx_i^\ep),\nu\right)}^2\,d\sigma<1\,,
    \end{equation}
    when $\ep$ is sufficiently small.

    2. {\em Boundedness from below}. When $g\geq 0$, the result is
    automatic as the functional $\hat E^i_\ep$ is nonnegative. Suppose
    that $g<0$. Let $\ep>0$ be small enough so that $\calP_i\subset
    B_R({\bf 0})$ and choose $\bu\in C_0^\infty\left(B_R({\bf
        0})\right)$ such that the support of $\bu$ is contained in
    $B_R({\bf 0})$.  Following the same line of reasoning as in the
    proof of Lemma \ref{l1} and switching to spherical coordinates
    with $z-$axis along the long axis of the spheroid $\calP_i$, we
    have
    \[
    \bu(\rho(\phi),\theta,\phi)=-\int_{\rho(\phi)}^R\bu_r(r,\theta,\phi)\,dr\,,
    \]
    where
    \begin{equation}
      \label{rho}
      \rho(\phi)=\frac{ab}{{\left(b^2{\sin^2}\phi+a^2{\cos^2}\phi\right)}^\frac{1}{2}}
    \end{equation}
    is the equation of the spheroid. By H\"older's inequality
    \[
    \begin{split}
      {\left(\int_{\rho(\phi)}^R\bu_r(r,\theta,\phi)\,dr\right)}^2&\leq\int_{\rho(\phi)}^R{\left|\bu_r(r,\theta,\phi)\right|}^2r^2\,dr\int_{\rho(\phi)}^Rr^{-2}\,dr\\
      &
      \leq\frac{1}{\rho(\phi)}\int_{\rho(\phi)}^R{\left|\bu_r(r,\theta,\phi)\right|}^2r^2\,dr\,,
    \end{split}
    \]
    then
    \begin{equation}
      \label{l3e1}
      {\left|\bu(\rho(\phi),\theta,\phi)\right|}^2\leq\frac{1}{\rho(\phi)}\int_{\rho(\phi)}^R{\left|\bu_r(r,\theta,\phi)\right|}^2r^2\,dr\,.
    \end{equation}
    For the prolate spheroid with long axis in the direction of
    $z$-axis, the element of the surface area is given by
    \begin{equation}
      \label{surf_jac}
      d\sigma={\left(\rho^2+\rho_\phi^2\right)}^\frac{1}{2}\rho\sin{\phi}\,d\theta\,d\phi\,.
    \end{equation}
    Multiplying (\ref{l3e1}) by the Jacobian and integrating, we
    obtain
    \[
    \begin{split}
      \int_{0}^{\pi}\int_{0}^{2\pi}&{\left|\bu(\rho(\phi),\theta,\phi)\right|}^2{\left(\rho^2+\rho_\phi^2\right)}^{1/2}\rho\sin\phi\,d\theta\,d\phi \\
      &\leq\int_{0}^{\pi}\int_{0}^{2\pi}\int_{\rho(\phi)}^R{\left|\bu_r(r,\theta,\phi)\right|}^2{\left(\rho^2+\rho_\phi^2\right)}^{1/2}r^2\sin\phi\,dr\,d\theta\,d\phi\,,
    \end{split}
    \]
    then
    \begin{equation}
      \label{l3e2}
      \begin{split}
        \int_{\partial\calP_i}{|\bu|}^2\,d\sigma&\leq \max_{\phi\in[0,\pi]}{\left(\rho^2+\rho_\phi^2\right)}^{1/2}\int_{B_R({\bf 0})\backslash\calP_i}{|\nabla\bu|}^2\,dV \\
        &\leq C\int_{B_R({\bf
            0})\backslash\calP_i}{|\nabla\bu|}^2\,dV\,,
      \end{split}
    \end{equation}
    where the constant $C>0$ depends only on $\calP_i$.

    Using (\ref{l3e2}) we obtain the following estimate
    \begin{equation}
      \label{l3e3}
      \begin{split}
        &\hat E_\ep^i[\bu]=\int_{B_R({\bf{0}})\backslash\calP_i}\left[{|\nabla\bu|}^2+{|\bu|}^2\right]\,dV+g\ep^{3-\alpha}\int_{\partial\calP_i}{(\bu+\bw_i,\nu)}^2\,d\sigma \\
        &\geq\int_{B_R({\bf{0}})\backslash\calP_i}\left[{|\nabla\bu|}^2+{|\bu|}^2\right]\,dV-2|g|\ep^{3-\alpha}\left[\int_{\partial\calP_i}{(\bu,\nu)}^2\,d\sigma\right.\\
        &\left.+\int_{\partial\calP_i}{(\bw_i,\nu)}^2\,d\sigma\right]= \left(1-C\ep^{3-\alpha}\right)\int_{B_R({\bf{0}})\backslash\calP_i}{|\nabla\bu|}^2\,dV \\
        &+\int_{B_R({\bf 0})\backslash\calP_i}{|\bu|}^2\,dV-2|g|\ep^{3-\alpha}\int_{\partial\calP_i}{(\bw_i,\nu)}^2\,d\sigma \\
        &\geq\frac{1}{2}\int_{B_R({\bf{0}})\backslash\calP_i}\left[{|\nabla\bu|}^2+{|\bu|}^2\right]\,dV-1\,,
      \end{split}
    \end{equation}
    when $\ep$ is sufficiently small uniformly in $\bu$. It
    follows that
    \[\hat E_\ep^i[\bu]>-1\,,\]
    for the same values of $\ep$. Since
    $C_0^\infty\left(B_R({\bf{0}})\backslash\calP_i\right)$ is dense
    in $H^1_0\left(B_R({\bf{0}})\backslash\calP_i\right)$, the
    inequalities (\ref{l3e2}) and (\ref{l3e3}) hold for all $\bu\in
    H^1\left(B_R({\bf{0}})\backslash\calP_i\right)$.

    3. {\em Existence of a minimizer}. Suppose that
    $\left\{\bu_k\right\}\subset
    H^1_0\left(B_R({\bf{0}})\backslash\calP_i\right)$ is a minimizing
    sequence for $\hat E_\ep^i$. For a sufficiently small $\ep$, from
    (\ref{l3e0}) and (\ref{l3e3}) we can assume that
    \begin{equation}
      \label{l3e3.1}
      \int_{B_R({\bf{0}})\backslash\calP_i}\left[{|\nabla\bu|}^2+{|\bu|}^2\right]\,dV<2\,,
    \end{equation}
    uniformly in $k$. Then, up to a subsequence,
    $\left\{\bu_k\right\}$ converges weakly in the space
    $H^1_0\left(B_R({\bf{0}})\backslash\calP_i\right)$ to a
    $\hat\bu^\ep_i$ that minimizes $\hat E_\ep^i$ by the lower
    semicontinuity of (\ref{local_func}) and the trace theorem.

    4. {\em Properties of the minimizer}. In this part of the proof,
    $C$ denotes various constants that depend on $\calP_i$ and $\bw_i$
    only.  Multiplying the equation (\ref{uind}) by $\hat\bu_i^{\ep}$
    and integrating by parts over $B_R({\bf 0})\backslash\calP_i$, we
    have
  \begin{equation}
    \label{l3e5}
    \begin{split}
    \int_{B_R({\bf 0})\backslash\calP_i}&\left[{\left|\nabla\hat\bu_i^{\ep}\right|}^2+{\left|\hat\bu_i^{\ep}\right|}^2\right]\,dV \\
    &=-g\ep^{3-\alpha}\int_{\partial\calP_i}\left(\hat\bu_i^{\ep}+\bw_i,\nu\right)\left(\hat\bu_i^{\ep},\nu\right)\,d\sigma\,.
    \end{split}
  \end{equation}
  From (\ref{l3e3}) and H\"older's inequality it follows that
  \begin{equation}
    \label{l3e6}
    \begin{split}
    \int_{\partial\calP_i}&\left(\hat\bu_i^{\ep}+\bw_i,\nu\right)\left(\hat\bu_i^{\ep},\nu\right)\,d\sigma\leq\int_{\partial\calP_i}{\left(\hat\bu_i^{\ep},\nu\right)}^2\,d\sigma+C{\left(\int_{\partial\calP_i}{\left(\hat\bu_i^{\ep},\nu\right)}^2\,d\sigma\right)}^{1/2} \\
    &\leq C\left[\int_{B_R({\bf 0})\backslash\calP_i}{\left|\nabla\hat\bu_i^{\ep}\right|}^2\,dV+{\left(\int_{B_R({\bf 0})\backslash\calP_i}{\left|\nabla\hat\bu_i^{\ep}\right|}^2\,dV\right)}^{1/2}\right]\,,
    \end{split}
  \end{equation}
  when $\ep$ is small enough. Now, combining (\ref{l3e5}) and
  (\ref{l3e6}) we obtain that
  \begin{align}
    \label{l3e7}
\begin{split}
  \int_{B_R({\bf
      0})\backslash\calP_i}&{\left|\nabla\hat\bu_i^{\ep}\right|}^2\,dV\leq
  C\ep^{3-\alpha}\left[\int_{B_R({\bf
        0})\backslash\calP_i}{\left|\nabla\hat\bu_i^{\ep}\right|}^2\,dV\right.
  \\ &\left.+{\left(\int_{B_R({\bf
            0})\backslash\calP_i}{\left|\nabla\hat\bu_i^{\ep}\right|}^2\,dV\right)}^{1/2}\right]\,,
\end{split} \\
    \label{l3e8}
\begin{split}
  \int_{B_R({\bf
      0})\backslash\calP_i}&{\left|\hat\bu_i^{\ep}\right|}^2\,dV\leq
  C\ep^{3-\alpha}\left[\int_{B_R({\bf
        0})\backslash\calP_i}{\left|\nabla\hat\bu_i^{\ep}\right|}^2\,dV\right.
  \\ &\left.+{\left(\int_{B_R({\bf
            0})\backslash\calP_i}{\left|\nabla\hat\bu_i^{\ep}\right|}^2\,dV\right)}^{1/2}\right]\,.
\end{split}
  \end{align}
  From (\ref{l3e7}), we find that
\[
  \int_{B_R({\bf 0})\backslash\calP_i}{\left|\nabla\hat\bu_i^{\ep}\right|}^2\,dV\leq C\ep^{6-2\alpha}\,,
\]
and then, from (\ref{l3e8})
\[
  \int_{B_R({\bf 0})\backslash\calP_i}{\left|\hat\bu_i^{\ep}\right|}^2\,dV\leq C\ep^{6-2\alpha}\,,
\]
uniformly in $\ep\ll 1$. Finally, (\ref{l33}) follows from (\ref{l31})
and (\ref{l3e2}).
  \end{proof}
  Recall that $R=\ep^{\kappa-\alpha}$. Rewriting (\ref{l31}-\ref{l33}) in terms of $\bx$ gives
    \begin{align}
      \label{l31r}
      &\int_{B_{\ep^\kappa}(\bx_i^\ep)\backslash\calP_i^\ep}{\left|\nabla\bu_i^{\ep}\right|}^2\,dV\leq C\ep^{6-\alpha}\,,\\
      \label{l32r}
      &\int_{B_{\ep^\kappa}(\bx_i^\ep)\backslash\calP_i^\ep}{\left|\bu_i^{\ep}-\bw_i\right|}^2\,dV\leq C\ep^{6+\alpha}\,,\\
      \label{l33r}
      &\int_{\partial\calP_i^\ep}{\left|\bu_i^{\ep}-\bw_i\right|}^2\,d\sigma\leq
      C\ep^6\,,
    \end{align}
when $\ep$ is sufficiently small. Furthermore
\begin{equation}
  \label{bdrest}
  \begin{split}
    &g_\ep\int_{\partial\calP_i^\ep}{\left(\bu_i^\ep,\nu\right)}^2\,d\sigma=g\ep^{3-2\alpha}\int_{\partial\calP_i^\ep}{\left\{\left(\bu_i^\ep-\bw_i,\nu\right)+\left(\bw_i,\nu\right)\right\}}^2\,d\sigma \\
    &=g\ep^{3-2\alpha}\left[\int_{\partial\calP_i^\ep}{\left(\bw_i,\nu\right)}^2\,d\sigma+2\int_{\partial\calP_i^\ep}{\left(\bw_i,\nu\right)}{\left(\bu_i^\ep-\bw_i,\nu\right)}\,d\sigma\right. \\
    &+\left.\int_{\partial\calP_i^\ep}{\left(\bu_i^\ep-\bw_i,\nu\right)}^2\,d\sigma\right]\,.
  \end{split}
\end{equation}
By H\"older's inequality, (\ref{l33r}), and the fact that $\bw\in
C^\infty(\bar\Omega)$, we have
\begin{equation}
  \label{bdrest1}
\begin{split}
  \int_{\partial\calP_i^\ep}&{\left(\bw_i,\nu\right)}{\left(\bu_i^\ep-\bw_i,\nu\right)}\,d\sigma \\ &\leq{\left(\int_{\partial\calP_i^\ep}{\left(\bw_i,\nu\right)}^2\,d\sigma\right)}^{1/2}{\left(\int_{\partial\calP_i^\ep}{\left(\bu_i^\ep-\bw_i,\nu\right)}^2\,d\sigma\right)}^{1/2} \\
  &\leq{\left(\int_{\partial\calP_i^\ep}{\left(\bw_i,\nu\right)}^2\,d\sigma\right)}^{1/2}{\left(\int_{\partial\calP_i^\ep}{\left|\bu_i^\ep-\bw_i\right|}^2\,d\sigma\right)}^{1/2}=O\left(\ep^4\right)\,.
\end{split}
\end{equation}
Since the last integral in (\ref{bdrest}) is $O\left(\ep^6\right)$, we conclude that
\begin{equation}
  \label{bdrest2}
  g_\ep\int_{\partial\calP_i^\ep}{\left(\bu_i^\ep,\nu\right)}^2\,d\sigma=g\ep^{3-2\alpha}\int_{\partial\calP_i^\ep}{\left(\bw_i,\nu\right)}^2\,d\sigma+O\left(\ep^{7-2\alpha}\right)\,.
\end{equation}

We now return to estimating $\mathcal E_\ep\left[\bw^\ep\right]$. From
(\ref{testfn}) we have
\[
\nabla\bw^\ep=\nabla\bw+\nabla\bz^\ep\,,
\]
where
\begin{equation}
  \label{testfnder}
  \begin{split}
    \nabla\bz^\ep&=\sum_i\left\{\phi\left(\ep^{-\kappa}\left|\bx-\bx_i^\ep\right|\right)\nabla\left(\bu_i^\ep-\bw\right)\right.\\
    &\left.+\frac{1}{\ep^\kappa}\phi^\prime\left(\ep^{-\kappa}\left|\bx-\bx_i^\ep\right|\right)\frac{\bx-\bx_i^\ep}{\left|\bx-\bx_i^\ep\right|}\otimes\left(\bu_i^\ep-\bw\right)\right\}\,.
  \end{split}
\end{equation}
Then, since the supports of
$\phi\left(\ep^{-\kappa}\left|\bx-\bx_i^\ep\right|\right)$ and
$\phi\left(\ep^{-\kappa}\left|\bx-\bx_j^\ep\right|\right)$ are
mutually nonintersecting for any $i\neq j\in{1,\ldots,N_\ep}$, using
the definition of $\phi$, we have
\begin{equation}
  \label{dz1}
  \begin{split}
    &\int_{\Omega\backslash\cup_i\calP_i^\ep}{\left|\nabla\bz_i^\ep\right|}^2\,dV\leq 2\sum_i\int_{\Omega\backslash\calP_i^\ep}\phi^2\left(\ep^{-\kappa}\left|\bx-\bx_i^\ep\right|\right){\left|\nabla\left(\bu_i^\ep-\bw\right)\right|}^2\,dV \\
    &+\frac{2}{\ep^{2\kappa}}\sum_i\int_{\Omega\backslash\calP_i^\ep}{\left[\phi^\prime\left(\ep^{-\kappa}\left|\bx-\bx_i^\ep\right|\right)\right]}^2{\left|\bu_i^\ep-\bw\right|}^2\,dV \\
    &\leq C\sum_i\int_{B_{\ep^\kappa}\left(\bx_i^\ep\right)\backslash\calP_i^\ep}\left[{\left|\nabla\left(\bu_i^\ep-\bw\right)\right|}^2+\frac{1}{\ep^{2\kappa}}{\left|\bu_i^\ep-\bw\right|}^2\right]\,dV\,,
  \end{split}
\end{equation}
where $C$ depends on $\phi$ only. Since $\bw\in C^\infty(\bar\Omega)$, the following estimates hold
\begin{align}
  \label{wap1}
  \begin{split}
    {\left|\bu_i^\ep(\bx)-\bw(\bx)\right|}^2&\leq 2{\left|\bu_i^\ep(\bx)-\bw_i\right|}^2+2{\left|\bw(\bx)-\bw_i\right|}^2 \\ 
    &\leq C\left[{\left|\bu_i^\ep(\bx)-\bw_i\right|}^2+{\left|\bx-\bx_i^\ep\right|}^2\right]\,,
  \end{split} \\
  \label{wap2}
  &{\left|\nabla\left(\bu_i^\ep(\bx)-\bw(\bx)\right)\right|}^2\leq
  2{\left|\nabla\bu_i^\ep(\bx)\right|}^2+C\,,
\end{align}
for every $\bx\in
B_{\ep^\kappa}\left(\bx_i^\ep\right)\backslash\calP_i^\ep,$ where
$C>0$ is a constant that depends on $\bw$ only. Therefore, by (\ref{l31r}) and (\ref{l32r}) we obtain
\begin{equation}
  \label{wap3}
  \begin{split}
    &\int_{B_{\ep^\kappa}\left(\bx_i^\ep\right)\backslash\calP_i^\ep}{\left|\nabla\left(\bu_i^\ep-\bw\right)\right|}^2\,dV \\ &\leq 2\int_{B_{\ep^\kappa}\left(\bx_i^\ep\right)\backslash\calP_i^\ep}{\left|\nabla\bu_i^\ep\right|}^2\,dV+C\left|B_{\ep^\kappa}\left(\bx_i^\ep\right)\right|=O(\ep^{\min\left\{6-\alpha,3\kappa\right\}})\,,
  \end{split}
\end{equation}
and
\begin{equation}
  \label{wap4}
  \begin{split}
    \frac{1}{\ep^{2\kappa}}&\int_{B_{\ep^\kappa}\left(\bx_i^\ep\right)\backslash\calP_i^\ep}{\left|\bu_i^\ep-\bw\right|}^2\,dV \\
    &\leq
    C\left[\frac{1}{\ep^{2\kappa}}\int_{B_{\ep^\kappa}\left(\bx_i^\ep\right)\backslash\calP_i^\ep}{\left|\bu_i^\ep-\bw_i\right|}^2\,dV+\ep^{3\kappa}\right]=O\left(\ep^{\min\{6+\alpha-2\kappa,3\kappa\}}\right)
\end{split}
\end{equation}
Here $\left|B_{\ep^\kappa}\left(\bx_i^\ep\right)\right|$ is the volume of $B_{\ep^\kappa}\left(\bx_i^\ep\right)$. It follows that
\begin{equation}
  \label{zest}
  \int_{\Omega\backslash\cup_i\calP_i^\ep}{\left|\nabla\bz_i^\ep\right|}^2\,dV=O\left(\ep^{\min\{3-\alpha,3(\kappa-1)\}}\right)=o(1)\,,
\end{equation}
since $1<\alpha<2$, $1<\kappa<\alpha$, and there are
$O\left(\ep^{-3}\right)$ spheroidal particles. In addition, by
H\"older's inequality,
\begin{equation}
  \label{gradconv}
  \begin{split}
    \int_{\Omega\backslash\cup_i\calP_i^\ep}&{\left|\nabla\bw^\ep\right|}^2\,dV=\int_{\Omega\backslash\cup_i\calP_i^\ep}{\left|\nabla\bz_i^\ep+\nabla\bw\right|}^2\,dV \\
    &=\int_{\Omega\backslash\cup_i\calP_i^\ep}{\left|\nabla\bw\right|}^2\,dV+o(1)=\int_{\Omega}{\left|\nabla\bw\right|}^2\,dV+O\left(\ep^{3(\alpha-1)}\right)+o(1) \\
    &=\int_{\Omega}{\left|\nabla\bw\right|}^2\,dV+o(1)\,,
  \end{split}
\end{equation}
when $\ep$ is small. This result extends to $\bw\in H^1(\Omega)$ by a density argument.

Next, consider the asymptotic behavior of the nonlinear term.
Extending continuously $\bw^\ep$ to $\tilde\bw^\ep\in H^1(\Omega)$ and
using the uniform boundedness of $\bw_\ep$ in $H^1(\Omega)$ (e.g. from
(\ref{gradconv}) and Poincare's inequality), we conclude that there is
a subsequence such that $\tilde \bw^\ep\rightharpoonup\bw$ weakly in
$H^1(\Omega)$ and strongly in $L^p(\Omega)$ where $1<p<6$. Since the
Lebesgue measure of the set $\cup_i\calP_i^\ep$ converges to zero when
$\ep\to 0$ and $\bw\in C^\infty(\bar\Omega)$, we have that
\begin{equation}
  \label{nonlconv}
  \int_{\Omega\backslash\cup_i\calP_i^\ep}{\left(1-\left|\bw^\ep\right|^2\right)}^2\,dV\to\int_{\Omega}{\left(1-\left|\bw\right|^2\right)}^2\,dV\,,
\end{equation}
as $\ep\to 0$.

Finally, by (\ref{bdrest2}), we determine that
\begin{equation}
  \label{bdryconv}
  \begin{split}
    g\ep^{3-2\alpha}&\sum_i\int_{\partial\calP_i^\ep}{\left(\bu_i^\ep,\nu\right)}^2\,d\sigma=g\ep^{3-2\alpha}\sum_i\int_{\partial\calP_i^\ep}{\left(\bw_i,\nu\right)}^2\,d\sigma+O\left(\ep^{4-2\alpha}\right)\\
    &=g\ep^{3-2\alpha}\sum_i\int_{\partial\calP_i^\ep}{\left(\bw_i,\nu\right)}^2\,d\sigma+o(1) \\
    &=g\ep^3\sum_i\int_{\partial\calP_i}{\left(\bw_i,\nu\right)}^2\,d\sigma+o(1)\,,
  \end{split}
\end{equation}
since $\alpha<2$. Thus
\begin{equation}
  \label{enconv}
  \begin{split}
    \mathcal E_\ep\left[\bw^\ep\right]&=\int_{\Omega}\left[{\left|\nabla\bw\right|}^2+{\left(1-\left|\bw\right|^2\right)}^2\right]\,dV \\ &+g\ep^3\sum_i\int_{\partial\calP_i}{\left(\bw_i,\nu\right)}^2\,d\sigma+o(1)\,,
  \end{split}
\end{equation}
when $\ep$ is small. It remains to determine the asymptotic limit of
the boundary term as $\ep\to 0$. The sum in this term can be rewritten
as follows
\begin{equation}
  \label{bdry1}
  \begin{split}
    g\ep^3\sum_i&\int_{\partial\calP_i}{\left(\bw\left(\bx_i^\ep\right),\nu\right)}^2\,d\sigma=g\ep^3\sum_i\int_{\partial\calP_i}{\left(w_k\left(\bx_i^\ep\right){\bf
          e}_k,\nu\right)}^2\,d\sigma \\
    &=g\ep^3\sum_i\left[\int_{\partial\calP_i}{\left({\bf
            e}_k,\nu\right)}{\left({\bf
            e}_j,\nu\right)}\,d\sigma\right]w_j\left(\bx_i^\ep\right)w_k\left(\bx_i^\ep\right) \\
    &=g\ep^3\sum_i\left[\int_{\partial\calP}{\left({\bf
            e}_k,R_i^\ep\nu\right)}{\left({\bf
            e}_j,R_i^\ep\nu\right)}\,d\sigma\right]w_j\left(\bx_i^\ep\right)w_k\left(\bx_i^\ep\right) \\
    &=\int_\Omega\left(A^\ep\left(\bx\right)\bw\left(\bx\right),\bw\left(\bx\right)\right)\,dV\,,
  \end{split}
\end{equation}
where ${\bf e}_k,\ k=1,2,3$ is an orthonormal basis in $\mathbb R^3$,
the matrix-valued function
$A^\ep_{jk}\left(\bx\right)=g\ep^3\sum_i\delta\left(\bx-\bx_i^\ep\right)\int_{\partial\calP}{\left({\bf
      e}_k,R_i^\ep\nu\right)}{\left({\bf e}_j,R_i^\ep\nu\right)}\,d\sigma$,
the function $R_i^\ep\in M^{3\times 3},\ i=1,\ldots,N_\ep$ is a rotation
matrix, such that $\calP_i=R_i^\ep\calP$.  Further,
$w_k=\left(\bw\left(\bx_i^\ep\right),{\bf e}_k\right)$, where $k=1,2,3$
and we assume summation over the repeated indices. Thus, from our
assumptions on the geometry of the domain
\begin{equation}
  \label{greal}
  \mathcal E_\ep\left[\bw^\ep\right]\to \int_{\Omega}\left[{\left|\nabla\bw\right|}^2+{\left(1-\left|\bw\right|^2\right)}^2+\left(A\bw,\bw\right)\right]\,dV\,,
\end{equation}
for every $\bw\in C^\infty(\bar\Omega)$.
\end{proof}

\begin{theorem}
  Let a sequence of minimizers $\{\bu_\ep\}$ of $\mathcal E_\ep$ be
  such that the sequence $\{\tilde\bu_\ep\}$ of extensions of
  $\{\bu_\ep\}$ to $\Omega$ converges weakly in $H^1(\Omega)$ to some
  $\bu\in H^1(\Omega)$. Then
  \begin{equation}
    \label{lowsem}
    \liminf_{\ep\to 0}\mathcal E_\ep[\bu_\ep]\geq \mathcal E[\bu]\,,
  \end{equation}
  where $\mathcal E$ is defined by (\ref{limfun}).
\end{theorem}

\begin{proof}
  Suppose that there is $\left\{\bu_\delta\right\}\subset C^1(\Omega)$
  such that $\bu_\delta\to\bu$ strongly in $H^1(\Omega)$ and the
  extensions to $\Omega$ of minimizers $\bu_\ep$ of $\mathcal E_\ep$
  converge $\tilde\bu_\ep\rightharpoonup\bu$ weakly in $H^1(\Omega)$.
  We construct $\bu_\delta^\ep=\bu_\delta+\bz_\delta^\ep$ in the same
  way as in (\ref{testfn}), so that their extensions $\tilde
  \bu_\delta^\ep\rightharpoonup\bu_\delta$ converge weakly in
  $H^1(\Omega)$ along with $\mathcal
  E_\ep\left[\bu^\ep_\delta\right]\to\mathcal
  E\left[\bu_\delta\right]$ as $\ep\to 0$. Let
  $\tilde\zeta^\ep_\delta:=\tilde\bu_\ep-\tilde\bu_\delta^\ep$ and
  denote its restriction to $\Omega\backslash\cup_i\calP_i^\ep$ by
  $\zeta_\ep^\delta$. Then
  $\tilde\zeta_\ep^\delta\rightharpoonup\zeta_\delta:=\bu-\bu_\delta\,,$
  weakly in $H^1(\Omega)$ and strongly in $L^p(\Omega)$ for $p<6$ as
  $\ep\to 0$.

  1. {\em{Asymptotics of $\zeta_\ep^\delta$}}. First, we show that
  \begin{equation}
    \label{eq:zetaclaim}
    \lim_{\delta\to 0}\limsup_{\ep\to 0}\left\|\nabla\zeta_\ep^\delta\right\|_{L^2\left(\Omega\backslash\cup_i\calP_i^\ep\right)}=0\,.
  \end{equation}
  We begin by observing that the expression for $\mathcal
  E_\ep[\bu_\ep]$ can be rewritten so that
  \begin{equation}
    \label{eq:long}
    \begin{split}
      \mathcal E_\ep&[\bu_\ep]=\mathcal
      E_\ep[\bu_\ep^\delta]+\int_{\Omega\backslash\cup_i\calP_i^\ep}{\left|\nabla\zeta_\ep^\delta\right|}^2\,dV+2\int_{\Omega\backslash\cup_i\calP_i^\ep}\left(\nabla\zeta_\ep^\delta,\nabla\bu_\ep^\delta\right)\,dV \\
      &+\int_{\Omega\backslash\cup_i\calP_i^\ep}{\left|\zeta_\ep^\delta\right|}^4\,dV-2\int_{\Omega\backslash\cup_i\calP_i^\ep}{\left|\zeta_\ep^\delta\right|}^2\,dV-4\int_{\Omega\backslash\cup_i\calP_i^\ep}\left(\zeta_\ep^\delta,\bu_\ep^\delta\right)\,dV \\
      &+2\int_{\Omega\backslash\cup_i\calP_i^\ep}{\left|\zeta_\ep^\delta\right|}^2{\left|\bu_\ep^\delta\right|}^2\,dV+4\int_{\Omega\backslash\cup_i\calP_i^\ep}\left(\zeta_\ep^\delta,\bu_\ep^\delta\right)^2\,dV \\
      &+4\int_{\Omega\backslash\cup_i\calP_i^\ep}{\left|\bu_\ep^\delta\right|}^2\left(\zeta_\ep^\delta,\bu_\ep^\delta\right)\,dV+4\int_{\Omega\backslash\cup_i\calP_i^\ep}{\left|\zeta_\ep^\delta\right|}^2\left(\zeta_\ep^\delta,\bu_\ep^\delta\right)\,dV \\
      &+2\sum_ig_\ep\int_{\partial\calP_i^\ep}{\left(\zeta_\ep^\delta,\nu\right)}^2\,d\sigma+\sum_ig_\ep\int_{\partial\calP_i^\ep}\left(\zeta_\ep^\delta,\nu\right)\left(\bu_\ep^\delta,\nu\right)\,d\sigma\,.
    \end{split}
  \end{equation}
Since $\bu_\ep$ is a minimizer of $\mathcal E_\ep$, we have that
\[\mathcal E_\ep[\bu_\ep]\leq\mathcal E_\ep[\bu^\delta_\ep]\,,\]
then
  \begin{equation}
    \label{eq:longzeta}
    \begin{split}
      \int_{\Omega\backslash\cup_i\calP_i^\ep}&{\left|\nabla\zeta_\ep^\delta\right|}^2\,dV\leq -2\int_{\Omega\backslash\cup_i\calP_i^\ep}\left(\nabla\zeta_\ep^\delta,\nabla\bu_\ep^\delta\right)\,dV \\
      &+2\int_{\Omega\backslash\cup_i\calP_i^\ep}{\left|\zeta_\ep^\delta\right|}^2\,dV+4\int_{\Omega\backslash\cup_i\calP_i^\ep}\left(\zeta_\ep^\delta,\bu_\ep^\delta\right)\,dV \\
      &-4\int_{\Omega\backslash\cup_i\calP_i^\ep}{\left|\bu_\ep^\delta\right|}^2\left(\zeta_\ep^\delta,\bu_\ep^\delta\right)\,dV-4\int_{\Omega\backslash\cup_i\calP_i^\ep}{\left|\zeta_\ep^\delta\right|}^2\left(\zeta_\ep^\delta,\bu_\ep^\delta\right)\,dV \\
      &-2\sum_ig_\ep\int_{\partial\calP_i^\ep}{\left(\zeta_\ep^\delta,\nu\right)}^2\,d\sigma-\sum_ig_\ep\int_{\partial\calP_i^\ep}\left(\zeta_\ep^\delta,\nu\right)\left(\bu_\ep^\delta,\nu\right)\,d\sigma\,.
    \end{split}
  \end{equation}
  We need to estimate each term on the right hand side of
  (\ref{eq:longzeta}). In the remainder of the proof, $C>0$ denotes a
  constant independent of $\ep$ and $\delta$.

  (a). Beginning with the first term, we write
\begin{equation}
  \label{eq:1}
  \begin{split}
    \int_{\Omega\backslash\cup_i\calP_i^\ep}{\left(\nabla\bu^\ep_\delta,\nabla\zeta^\ep_\delta\right)}\,dV&=\int_{\Omega\backslash\cup_i\calP_i^\ep}{\left(\nabla\bu_\delta,\nabla\zeta^\ep_\delta\right)}\,dV \\
    &+\int_{\Omega\backslash\cup_i\calP_i^\ep}{\left(\nabla\bz^\ep_\delta,\nabla\zeta^\ep_\delta\right)}\,dV\,.
  \end{split}
\end{equation}
We have
\begin{equation}
  \label{eq:2}
  \int_{\Omega\backslash\cup_i\calP_i^\ep}{\left(\nabla\bu_\delta,\nabla\zeta^\ep_\delta\right)}\,dV=\int_{\Omega}{\left(\nabla\bu_\delta,\nabla\tilde\zeta^\ep_\delta\right)}\,dV-\int_{\cup_i\calP_i^\ep}{\left(\nabla\bu_\delta,\nabla\tilde\zeta^\ep_\delta\right)}\,dV\,.
\end{equation}
The second integral in (\ref{eq:2}) can be estimated with the help of
H\"older's and Minkowski's inequalities
\[
\begin{split}
  &\int_{\cup_i\calP_i^\ep}{\left(\nabla\bu_\delta,\nabla\tilde\zeta^\ep_\delta\right)}\,dV\leq{\left(\int_{\cup_i\calP_i^\ep}{\left|\nabla\bu_\delta\right|}^2\,dV\right)}^{1/2}{\left(\int_{\cup_i\calP_i^\ep}{\left|\nabla\tilde\zeta^\ep_\delta\right|}^2\,dV\right)}^{1/2} \\
  &\leq{\left(\int_{\cup_i\calP_i^\ep}{\left|\nabla\bu_\delta\right|}^2\,dV\right)}^{1/2}{\left(\int_\Omega\left[{\left|\nabla\tilde\bu^\ep_\delta\right|}^2+{\left|\nabla\tilde\bu_\ep\right|}^2\right]\,dV\right)}^{1/2} \\
  &\leq
  C{\left(\int_{\cup_i\calP_i^\ep}{\left|\nabla\bu_\delta\right|}^2\,dV\right)}^{1/2}\leq
  C{\left(\int_{\cup_i\calP_i^\ep}{\left|\nabla\bu_\delta-\nabla\bu+\nabla\bu\right|}^2\,dV\right)}^{1/2} \\
  &\leq C\left[{\left(\int_{\cup_i\calP_i^\ep}{\left|\nabla\bu\right|}^2\,dV\right)}^{1/2}+{\left(\int_{\cup_i\calP_i^\ep}{\left|\nabla\bu_\delta-\nabla\bu\right|}^2\,dV\right)}^{1/2}\right] \\
  &\leq C\left[{\left(\int_{\cup_i\calP_i^\ep}{\left|\nabla\bu\right|}^2\,dV\right)}^{1/2}+{\left(\int_{\Omega}{\left|\nabla\bu_\delta-\nabla\bu\right|}^2\,dV\right)}^{1/2}\right] \\
  &\to C{\left(\int_{\Omega}{\left|\nabla\bu_\delta-\nabla\bu\right|}^2\,dV\right)}^{1/2}\,,
\end{split}
\]
when $\ep\to 0$ because $\bu\in H^1(\Omega)$ and
$\left|\cup_i\calP_i^\ep\right|\to 0$. Here $C>0$ is independent of
$\delta$. Consider now the first integral in (\ref{eq:2}). By the weak
convergence of $\tilde\zeta^\ep_\delta$ to $\bu_\delta-\bu$ and
H\"older's inequality, we have that
\[
\begin{split}
  &\int_{\Omega}{\left(\nabla\bu_\delta,\nabla\tilde\zeta^\ep_\delta\right)}\,dV\to\int_{\Omega}{\left(\nabla\bu_\delta,\nabla(\bu_\delta-\bu)\right)}\,dV \\
  &\leq{\left(\int_\Omega{\left|\nabla\bu_\delta\right|}^2\right)}^{1/2}{\left(\int_\Omega{\left|\nabla\bu-\nabla\bu_\delta\right|}^2\right)}^{1/2} \\
  &\leq C\|\bu-\bu_\delta\|_{H^1(\Omega)}\,,
\end{split}
\]
when $\ep\to 0$. 

Now, for the second term in (\ref{eq:1}), we find
using (\ref{zest}) and the H\"older's inequality
\[
\begin{split}
\int_{\Omega\backslash\cup_i\calP_i^\ep}{\left(\nabla\bz^\ep_\delta,\nabla\zeta^\ep_\delta\right)}\,dV\leq{\left(\int_{\Omega\backslash\cup_i\calP_i^\ep}{\left|\nabla\bz^\ep_\delta\right|}^2\right)}^{1/2}{\left(\int_{\Omega\backslash\cup_i\calP_i^\ep}{\left|\nabla\zeta^\ep_\delta\right|}^2\right)}^{1/2}\to 0\,,
\end{split}
\]
when $\ep\to 0$.

(b). Consider the second term in (\ref{eq:longzeta}). Since
$\tilde\zeta_\ep^\delta$ converges weakly to $\bu-\bu_\delta$ in
$H^1(\Omega)$ and strongly in $L^p(\Omega)$ for $p<6$ when $\ep\to 0$,
we have that
\[
\begin{split}
  \int_{\Omega\backslash\cup_i\calP_i^\ep}&{\left|\zeta_\ep^\delta\right|}^2\,dV\leq\int_{\Omega}{\left|\tilde\zeta_\ep^\delta\right|}^2\,dV\to\int_\Omega{\left|\bu-\bu_\delta\right|}^2\,dV \\ 
  &\leq \|\bu-\bu_\delta\|_{H^1(\Omega)}\,,
\end{split}
\]
as $\ep\to 0$.

(c). Using H\"older's inequality, we find that
\[
\begin{split}
\left|\int_{\Omega\backslash\cup_i\calP_i^\ep}\left(\zeta_\ep^\delta,\bu_\ep^\delta\right)\,dV\right|\leq\left(\int_{\Omega}{\left|\tilde\bu_\ep^\delta\right|}^2\,dV\right)^{1/2}\left(\int_{\Omega}{\left|\tilde\zeta_\ep^\delta\right|}^2\,dV\right)^{1/2}\\
\to\left(\int_{\Omega}{\left|\bu_\delta\right|}^2\,dV\right)^{1/2}\left(\int_{\Omega}{\left|\bu-\bu_\delta\right|}^2\,dV\right)^{1/2}\leq C\|\bu-\bu_\delta\|_{H^1(\Omega)}\,,
\end{split}
\]
when $\ep\to 0$.

(d). Estimating in the same way as in (c), we obtain
\[
\begin{split}
  \left|\int_{\Omega\backslash\cup_i\calP_i^\ep}{\left|\bu_\ep^\delta\right|}^2\left(\zeta_\ep^\delta,\bu_\ep^\delta\right)\,dV\right|\leq\left(\int_{\Omega}{\left|\tilde\bu_\ep^\delta\right|}^4\,dV\right)^{3/4}\left(\int_{\Omega}{\left|\tilde\zeta_\ep^\delta\right|}^4\,dV\right)^{1/4}\\
  \to\left(\int_{\Omega}{\left|\bu_\delta\right|}^4\,dV\right)^{3/4}\left(\int_{\Omega}{\left|\bu-\bu_\delta\right|}^4\,dV\right)^{1/4}\leq
  C\|\bu-\bu_\delta\|_{H^1(\Omega)}\,,
\end{split}
\]
by Sobolev embedding.

(e). Estimating in the same way as in (c), we obtain
\[
\begin{split}
  \left|\int_{\Omega\backslash\cup_i\calP_i^\ep}{\left|\zeta_\ep^\delta\right|}^2\left(\zeta_\ep^\delta,\bu_\ep^\delta\right)\,dV\right|\leq\left(\int_{\Omega}{\left|\tilde\zeta_\ep^\delta\right|}^4\,dV\right)^{3/4}\left(\int_{\Omega}{\left|\tilde\bu_\ep^\delta\right|}^4\,dV\right)^{1/4}\\
  \to\left(\int_{\Omega}{\left|\bu_\delta\right|}^4\,dV\right)^{1/4}\left(\int_{\Omega}{\left|\bu-\bu_\delta\right|}^4\,dV\right)^{3/4}\leq
  C\|\bu-\bu_\delta\|^3_{H^1(\Omega)}\,,
\end{split}
\]
by Sobolev embedding.

(f). We use (\ref{bdrest3}) with $\lambda=1$ to obtain
\[
\begin{split}
  g_\ep&\sum_i\int_{\partial\calP_i^\ep}{\left(\zeta_\delta^\ep,\nu\right)}^2\,d\sigma \\
  & \leq C\left[\ep\int_{\omC}{\left|\nabla
        \zeta_\delta^\ep\right|}^2\,dV+\int_{\omC}{\left|\zeta_\delta^\ep\right|}^2\,dV\right] \\
  & \leq C\left[\ep\int_{\omC}{\left|\nabla
        \zeta_\delta^\ep\right|}^2\,dV+\int_{\Omega}{\left|\tilde\zeta_\delta^\ep\right|}^2\,dV\right] \\
  & \to C\int_{\Omega}{\left|\bu-\bu_\delta\right|}^2\,dV\leq
  C\|\bu-\bu_\delta\|_{H^1(\Omega)}\,,
\end{split}
\]
by the strong convergence of $\tilde\bu^\ep_\delta$ and
$\tilde\bu_\ep$ in $L^p(\Omega),\ 1<p<6,$ to $\bu_\delta$ and $\bu$,
respectively.

(g). Using the H\"older's inequality, we get
\[
\begin{split}
  &\left|\sum_ig_\ep\int_{\partial\calP_i^\ep}\left(\zeta_\ep^\delta,\nu\right)\left(\bu_\ep^\delta,\nu\right)\,d\sigma\right| \\
  &\leq\sum_i{\left(\int_{\partial\calP_i^\ep}\left|g_\ep\right|\left(\bu_\ep^\delta,\nu\right)^2\,dV\right)}^{1/2}{\left(\int_{\partial\calP_i^\ep}\left|g_\ep\right|\left(\zeta_\ep^\delta,\nu\right)^2\,dV\right)}^{1/2}\,.
\end{split}
\]
As in (f), applying (\ref{bdrest3}) with $\lambda=1$ we have that
\[\limsup_{\ep\to 0}\int_{\partial\calP_i^\ep}\left|g_\ep\right|\left(\bu_\ep^\delta,\nu\right)^2\,dV\leq  C\int_{\Omega}{\left|\bu_\delta\right|}^2\,dV\,,\]
then, with the help of (f) we obtain the estimate
\[
\begin{split}
  \limsup_{\ep\to 0}&\left|\sum_i g_\ep\int_{\partial\calP_i^\ep}\left(\zeta_\ep^\delta,\nu\right)\left(\bu_\ep^\delta,\nu\right)\,d\sigma\right| \\
  &\leq
  C{\left(\int_{\Omega}{\left|\bu_\delta\right|}^2\,dV\right)}^{1/2}{\left(\int_{\Omega}{\left|\bu-\bu_\delta\right|}^2\,dV\right)}^{1/2}\leq
  C\|\bu-\bu_\delta\|_{H^1(\Omega)}\,.
\end{split}
\]
Now, combining (a)-(f) and using (\ref{eq:longzeta}), leads to
\begin{equation}
  \label{eq:grzeta}
  \limsup_{\ep\to 0}\int_{\Omega\backslash\cup_i\calP_i^\ep}{\left|\nabla\zeta_\ep^\delta\right|}^2\,dV\leq C\|\bu-\bu_\delta\|_{H^1(\Omega)}\,,
\end{equation}
when $\delta$ is small. This, along with strong convergence of
$\bu_\delta$ to $\bu$ when $\delta\to 0$, proves the claim
(\ref{eq:zetaclaim}).

2. {\em{Limiting behavior of $\mathcal E_\ep[\bu_\ep]$}}. From (\ref{eq:long}) and (\ref{eq:longzeta}) we conclude that
  \begin{equation}
    \label{eq:longE}
    \begin{split}
      \mathcal E_\ep&[\bu_\ep]\geq \mathcal E_\ep[\bu_\ep^\delta]+2\int_{\Omega\backslash\cup_i\calP_i^\ep}\left(\nabla\zeta_\ep^\delta,\nabla\bu_\ep^\delta\right)\,dV \\
      &-2\int_{\Omega\backslash\cup_i\calP_i^\ep}{\left|\zeta_\ep^\delta\right|}^2\,dV-4\int_{\Omega\backslash\cup_i\calP_i^\ep}\left(\zeta_\ep^\delta,\bu_\ep^\delta\right)\,dV \\
      &+4\int_{\Omega\backslash\cup_i\calP_i^\ep}{\left|\bu_\ep^\delta\right|}^2\left(\zeta_\ep^\delta,\bu_\ep^\delta\right)\,dV+4\int_{\Omega\backslash\cup_i\calP_i^\ep}{\left|\zeta_\ep^\delta\right|}^2\left(\zeta_\ep^\delta,\bu_\ep^\delta\right)\,dV \\
      &+2\sum_ig_\ep\int_{\partial\calP_i^\ep}{\left(\zeta_\ep^\delta,\nu\right)}^2\,d\sigma+\sum_ig_\ep\int_{\partial\calP_i^\ep}\left(\zeta_\ep^\delta,\nu\right)\left(\bu_\ep^\delta,\nu\right)\,d\sigma \\
      &\geq\mathcal
      E_\ep[\bu_\ep^\delta]-\int_{\Omega\backslash\cup_i\calP_i^\ep}{\left|\nabla\zeta_\ep^\delta\right|}^2\,dV\,.
    \end{split}
  \end{equation}
Thus
\[
\begin{split}
  \liminf_{\ep\to 0}\mathcal E_\ep[\bu_\ep]\geq\lim_{\delta\to
    0}\liminf_{\ep\to 0}\mathcal E_\ep[\bu_\ep^\delta]-\lim_{\delta\to
    0}\limsup_{\ep\to
    0}\int_{\Omega\backslash\cup_i\calP_i^\ep}{\left|\nabla\zeta_\ep^\delta\right|}^2\,dV\,,
\end{split}
\]
and
\[\liminf_{\ep\to 0}\mathcal E_\ep[\bu_\ep]\geq \mathcal E[\bu]\,,\] 
because $\lim_{\ep\to 0}\mathcal E_\ep[\bu_\ep^\delta]=\mathcal
E[\bu_\delta]$ and $\mathcal E$ is continuous with respect to the
strong convergence of $\bu_\delta$ to $\bu$ in $H^1(\Omega)$.
\end{proof}

From (\ref{greal}) and (\ref{lowsem}) it follows that $\mathcal
E[\bu]\leq\mathcal E[\bw]$ for every $\bw\in H^1(\Omega)$, hence $\bu$
minimizes $\mathcal E$ over $H^1(\Omega)$.

\subsection{Magnetic Energy} Having established the asymptotics of the liquid crystalline component of the energy, we now turn our attention to magnetic interactions.  Consider (\ref{mag_energy}) for the prolate spheroidal particle $\calP$ with semiaxes $a>b$ and long axis oriented in the direction of $z$-axis. It is well known \cite{Landau-Lifshitz-ECM} that the solution to this problem in the exterior of $\calP$ is given by \begin{equation}
  \label{eq:magpot}
  \phi=\frac{4\pi a\,b^2m}{{\left(a^2-b^2\right)}^{3/2}}\left(\tanh^{-1}(t)-t\right)z\,,
\end{equation}
in cylindrical coordinates $(\rho,\theta,z),$ where
$t=\xi^{-1/2}{\left(a^2-b^2\right)}^{1/2}$ and $\xi$ is the largest
root of
\[\frac{\rho^2}{\xi+b^2-a^2}+\frac{z^2}{\xi}=1\,.\]
Further, $m$ is the density of the magnetic moment, so that
$\bm=\frac{4\pi a\,b^2 m}{3}\hat{\bf z}$ and $\hat{\bf z}$ is a unit
vector in the direction of $z$-axis. Assuming that
$a\ll{\left(\rho^2+z^2\right)}^{1/2}$ and expanding in
$a/{\left(\rho^2+z^2\right)}^{1/2}$, we find that
\begin{equation}
  \label{eq:exp_phi}\,
  \phi=\frac{32\pi a\,b^2m}{3r^3}z+O\left(z{(a/r)}^5\right),
\end{equation}
where $r=\sqrt{\rho^2+z^2}=|\bx|$. Note that the leading term in
(\ref{eq:exp_phi}) is identical to that for a sphere of the same
volume as $\calP$ and centered at the origin \cite{BorceaBruno}. The
leading order term in the expansion of the magnetic filed $\bf H$
generated by the ferromagnetic particle $\calP$ is given by
\[{\bf H}(\bx)=\frac{32\pi
  a\,b^2m}{3r^3}\left(\frac{3z}{r^2}{\bx}-\hat{\bf
    z}\right)+O\left((a/r)^5\right)\,,\] then
\begin{equation}
  \label{eq:exp_H}
  |{\bf H}(\bx)|=O\left(m(a/r)^3\right)\,,
\end{equation}
when $a/r\ll 1$.

Now consider the term corresponding to the magnetic interaction between the particles $\calP_i^\ep$ and $\calP_j^\ep$ for some $i,j=1\ldots,N_\ep$. We have
\begin{equation}
  \label{eq:twopart}
  \begin{split}
    \int_{\calP_i^\ep}&({\bf H}_j^\ep,\bm_i^\ep)\,dV+\int_{\calP_j^\ep}({\bf H}_i^\ep,\bm_j^\ep)\,dV=O\left(\frac{\left|m_i^\ep\right|\left|m_j^\ep\right|{\mathrm {Vol}}(\calP_i^\ep){\mathrm {Vol}}(\calP_j^\ep)}{d^3}\right) \\
    &=O\left(\ep^{6\alpha+2\beta_1-3}\right)\,,
\end{split}
\end{equation}
then
\begin{equation}
  \label{eq:mag_int_asym}
  \int_{\mathbb R^3}\left({\bf m_\ep,H_\ep}\right)\,dV=O\left(N^2_\ep\ep^{6\alpha+2\beta_1-3}\right)=O\left(\ep^{6\alpha+2\beta_1-9}\right)\to 0\,,
\end{equation}
when $\ep\to 0$ by our assumptions $\alpha$ and $\beta_1$. 

Finally, we consider the interaction between the external magnetic field and ferromagnetic particles. We have
\begin{equation}
  \label{eq:ext_magn}
  \begin{split}
  \int_{\mathbb R^3}&({\bf m}_\ep,{\bf h}_\ep)\,dV=\sum_i\int_{\calP_i^\ep}({\bf m}_i^\ep,{\bf h}_\ep)\,dV \\
  &=\int_\Omega\left({\bf h},{\bf M}^\ep\right)\,dV\to\int_\Omega\left({\bf h},{\bf M}\right)\,dV\,, \end{split} \end{equation}
by (\ref{asspars}) and (\ref{asstensor}), where ${\bf M}$ is the effective magnetic moment density. 

Combining the results for the liquid crystal and magnetic energies, we conclude that the minimizers of the family of functionals $\mathcal F_\ep$ converge to a minimizer of the functional 
\begin{equation}
\label{eq:limmagfun1}
  \mathcal F_0[\bu]=\int_{\Omega}\left[{\left|\nabla\bu\right|}^2+{\left(1-\left|\bu\right|^2\right)}^2+\left(A\bu,\bu\right)-2({\bf h},{\bf M})\right]\,dV,
\end{equation}
concluding the proof of Theorem \ref{t_main}.

\begin{remark}
 {\em  Suppose that the particles are distributed periodically in $\Omega$ with their centers of mass positioned at the vortices of a cubic lattice with the side $\epsilon$. If we assume that there exists a continuous function $R:\Omega\to {\mathrm{Orth}}^+:=\left\{X\in M^{3\times 3}:XX^T=I,\ \det{X}=1\right\}$ such that $R_i^\ep=R\left(\bx_i^\ep\right)$ for every $i=1\ldots N_\ep$ and $\ep>0$, then
\begin{equation}
  \label{eq:extereme_M}
  {\bf M}(\bx)=m\,R(\bx)\hat{\bf z}, 
\end{equation}
and
\begin{eqnarray}
  \label{eq:extereme_A}
  & A(\bx)=g\,R(\bx)\left(\int_{\partial\calP}\nu\otimes\nu d\sigma\right)R^T(\bx) & \nonumber \\ & =g\,R(\bx)(\lambda_1(\hat{\bf z}\otimes\hat{\bf z})+\lambda_2\left(I-\hat{\bf z}\otimes\hat{\bf z}\right)))R^T(\bx) & \\ & =\frac{g}{m^2}(\lambda_1({\bf M(x)}\otimes{\bf M(x)})+\lambda_2\left(I-{\bf M(x)}\otimes{\bf M(x)}\right)))\,, \nonumber &
\end{eqnarray}
where $\lambda_1$ and $\lambda_2$ are the two distinct eigenvalues of $\int_{\partial\calP}\nu\otimes\nu d\sigma$. The coupling terms in (\ref{eq:limmagfun1}) then take the form
 \begin{equation}
   \label{eq:density}
   \frac{g(\lambda_1-\lambda_2)}{m^2}\left({\bf M},\bu\right)^2+\frac{g\lambda_2}{m^2}{|\bu|}^2-2({\bf h},{\bf M}).
 \end{equation}
 For a needle-like prolate spheroid with a high aspect ratio we have that $\lambda_1\ll\lambda_2$ and the coefficient $\Lambda:=\frac{g(\lambda_1-\lambda_2)}{m^2}$ in front of $\left({\bf M},\bu\right)^2$ has a sign opposite to that of $g$. Hence nematic molecules align perpendicular to $\bf M$ when $\Lambda>0$ and parallel to $\bf M$ when $\Lambda<0$. Since the model in \cite{burylov-reiker95} assumes that $|\bu|=1$, the middle term in (\ref{eq:density}) can be neglected and the remaining interaction terms in (\ref{eq:density}) coincide with those in (\ref{burylov-energy}) up to a difference in notation.}  \end{remark}


\bibliography{ferro_lc}

\begin{thebibliography}{10}

\bibitem{BdG70}
F.~Brochard and P.~G. de~Gennes.
\newblock Theory of magnetic suspensions in liquid crystals.
\newblock {\em Le Journal de Physique}, 31:691, 1970.

\bibitem{RCB70}
J.~Rault, P.~E. Cladis, and J.~P. Burger.
\newblock Ferronematics.
\newblock {\em Physics Letters}, 32A:199--200, 1970.

\bibitem{Chen-Amer83}
S.~H. Chen and N.~M. Amer.
\newblock Observation of macroscopic collective behavior and new texture in
  magnetically doped liquid crystals.
\newblock {\em Phys. Rev. Lett.}, 51:2298--2301, 1983.

\bibitem{raikher86}
S.~V.~Burylov L.~Raikher and A.~N. Zakhlevnykh.
\newblock Orientational structure and magnetic properties of ferronematic in an
  external field.
\newblock {\em Zh.Eksp.Teor. Fiz}, 91:545--551, 1986.

\bibitem{burylov1-90}
S.~V. Burylov and L.~Raikher.
\newblock Ferronematics: On the development of the continuum theory approach.
\newblock {\em Journal of Magnetism and Magnetic Materials}, 85:74--76, 1990.

\bibitem{burylov2-90}
S.~V. Burylov and L.~Raikher.
\newblock Orientation of a solid particle embedded in a monodomain nematic
  liquid crystal.
\newblock {\em PRE}, 50:358--368, 1994.

\bibitem{burylov-reiker95}
S.~V. Burylov and L.~Raikher.
\newblock Macroscopic properties of ferronematics caused by orientational
  interactions on the particle surfaces. {I}, extended continuum model.
\newblock {\em Mol. Cryst.Liq.Cryst.}, 258:107--122, 1995.

\bibitem{raiker98}
L.~Raikher.
\newblock Orientational {F}rederiks-like transitions in ferronematics.
\newblock In P.~Toledano and A.M.Figueiredo Neto, editors, {\em Phase
  transitions in Complex Fluids}, pages 295--315, 1998.

\bibitem{bena2003}
R.~E. Bena and E.~Petrescu.
\newblock Surface effects on magnetic {F}reedericksz transition in
  ferronematics with soft particle anchoring.
\newblock {\em Journal of Magnetism and Magnetic Materials}, 263:353--359,
  2003.

\bibitem{Kopcansky2001}
P.~Kopcansky, M.~Koneracka, I.~Potocova, M.~Timko, J.~Jadzyn, and
  G.~Czechowski.
\newblock The anchoring of nematic molecules on magnetic particles in some
  types of ferronematics.
\newblock {\em Czech J. Phys.}, 51:59, 2001.

\bibitem{Kopcansky99}
I.~Potocova, P.~Kopcansky, M.~Koneracka, M.~Timko, J.~Jadzyn, and
  G.~Czechowski.
\newblock The influence of magnetic field on electric fredericksz transition in
  {8CB}-based ferronematic.
\newblock {\em Journal of Magnetism and Magnetic Materials}, 196-197:578--580,
  1999.

\bibitem{kopcansky2005}
P.~Kopcansky, I.~Potocova, M.~Koneracka, M.~Timko, A.~Jansen, J.~Jadzyn, and
  G.~Czechowski.
\newblock The anchoring of nematic molecules on magnetic particles in some
  types of ferroelectrics.
\newblock {\em Journal of Magnetism and Magnetic Materials}, 289:101--104,
  2005.

\bibitem{Z2006}
V.~I Zadorozhnii, A.~N. Vasilev, V.~Yu. Reshetnyak, K.~S. Thomas, and T.~J.
  Slukin.
\newblock Nematic director response in ferronematic cells.
\newblock {\em Europhys. Lett.}, 73:408--414, 2006.

\bibitem{selinger1}
L.~M. Lopatina and J.~V. Selinger.
\newblock Theory of ferroelectric nanoparticles in nematic liquid crystals.
\newblock {\em Phys. Rev. Lett.}, 102:197802, May 2009.

\bibitem{selinger2}
L.~M. Lopatina and J.~V. Selinger.
\newblock Maier-saupe-type theory of ferroelectric nanoparticles in nematic
  liquid crystals.
\newblock {\em Phys. Rev. E}, 84:041703, Oct 2011.

\bibitem{Landau-Lifshitz-ECM}
L.~D. Landau and E.~M. Lifshitz.
\newblock {\em Electrodynamics of Continuous Media : Volume 8 (Course of
  Theoretical Physics)}.
\newblock Butterworth-Heinemann, 1984.

\bibitem{virga}
Epifanio~G. Virga.
\newblock {\em Variational theories for liquid crystals}.
\newblock Chapman \& Hall, London, 1994.

\bibitem{mottram}
N.~Mottram and C.~Newton.
\newblock Introduction to {Q}-tensor theory.
\newblock Technical report, University of Strathclyde, 2004.

\bibitem{degennes}
P.~G. De~Gennes and J.~Prost.
\newblock {\em The physics of liquid crystals}.
\newblock Clarendon Press, Oxford, 1993.

\bibitem{BerKhrus}
L.~Berlyand and E.~Khruslov.
\newblock Competition between the surface and the boundary layer energies in a
  {G}inzburg-{L}andau model of a liquid crystal composite.
\newblock {\em Asymtotic Analysis}, 29:185--219, 2002.

\bibitem{Khruslov78}
E.~Khruslov.
\newblock Asymptotic behavior of the second boundary value problem under
  framentation of the domain's boundary.
\newblock {\em Mat. Sb.}, 106:604--621, 1978.

\bibitem{Berlyand83}
L.~Berlyand.
\newblock Homogenization of the {G}inzburg-{L}andau functional with surface
  energy term.
\newblock {\em Asymptotic Analysis}, 21:379--99, 1999.

\bibitem{BorceaBruno}
L.~Borcea and O.~Bruno.
\newblock On the magneto-elastic properties of elastomer-ferromagnet
  composites.
\newblock {\em Journal of the Mechanics and Physiscs of Solids}, 49:2877--2919,
  2001.

\end{thebibliography}

\end{document}